\documentclass[a4paper,12pt]{article}

\makeatletter
\@addtoreset{footnote}{page}
\makeatother

\usepackage{style-enumitem}

\usepackage{amsthm}
\usepackage{amsmath,amssymb,latexsym,amsfonts,mathrsfs}

\renewcommand{\hat}{\widehat}
\renewcommand{\tilde}{\widetilde}

\newcommand{\bD}{\ensuremath{\mathbb{D}}}
\newcommand{\bE}{\ensuremath{\mathbb{E}}}

\newcommand{\bP}{\ensuremath{\mathbb{P}}}

\newcommand{\bR}{\ensuremath{\mathbb{R}}}

\newcommand{\cB}{\ensuremath{\mathcal{B}}}

\newcommand{\cH}{\ensuremath{\mathcal{H}}}

\newcommand{\cL}{\ensuremath{\mathcal{L}}}
\newcommand{\cM}{\ensuremath{\mathcal{M}}}

\theoremstyle{plain}
\newtheorem{Thm}{Theorem}[section]

\newtheorem{Prop}[Thm]{Proposition}

\theoremstyle{definition}

\newtheorem{Def}[Thm]{Definition}
\newtheorem{Rem}[Thm]{Remark}

\numberwithin{equation}{section}

\makeatletter
\renewcommand\section{\@startsection {section}{1}{\z@}%
	{-3.5ex \@plus -1ex \@minus -.2ex}%
	{2.3ex \@plus.2ex}%
	{\normalfont\large\bf}}
\makeatother

\makeatletter
\renewcommand\subsection{\@startsection {subsection}{1}{\z@}%
	{-3.5ex \@plus -1ex \@minus -.2ex}%
	{2.3ex \@plus.2ex}%
	{\normalfont\normalsize\bf}}
\makeatother

\begin{document}

\begin{center}
	{\Large \bf 
		Fluctuation scaling limits for positive recurrent jumping-in diffusions with small jumps
	}
\end{center}
\begin{center}
	Kosuke Yamato and Kouji Yano
\end{center}
\begin{center}
	{\small \today}
\end{center}

\begin{abstract}
	For positive recurrent jumping-in diffusions with small jumps, we establish distributional limits of the fluctuations of inverse local times and occupation times.
	For this purpose, we introduce and utilize eigenfunctions with modified Neumann boundary condition and apply the Krein-Kotani correspondence. 
\end{abstract}


\section{Introduction}\label{section: intro}
 We study a strong Markov process $X$ on the half line $[0,\infty)$ (resp.\ the real line $\bR$) which has continuous paths up to the first hitting time of $0$ and, as soon as $X$ hits $0$, $X$ jumps into the interior $(0,\infty)$ (resp.\ $\bR \setminus \{0\}$) and starts afresh. 
 We call such a process $X$ a {\it unilateral} (resp.\ {\it bilateral}) {\it jumping-in diffusion}. 
 
 Let us consider the inverse local time $\eta$ at $0$ of a jumping-in diffusion $X$, especially in the positive recurrent case where the inverse local time $\eta$ has a degenerate scaling limit:
 \begin{align}
 	 \frac{1}{t}\eta (t) &\xrightarrow[t \to \infty]{P} b \in [0,\infty). \label{}
 \end{align}  
 One of our two main aims is to establish its {\it fluctuation scaling limit} of the inverse local time $\eta$ of the form:
 \begin{align}
 g(\gamma)\left(\frac{\eta(\gamma t)}{\gamma} - b t\right) &\xrightarrow[\gamma \to \infty]{d} S(t) \ \text{on} \  \bD  \label{eq92}
 \end{align}
 for some function $g(\gamma)$ which diverges to $\infty$ as $\gamma \to \infty$ and some stable process $S(t)$ without negative jumps. Here $\bD$ denotes the space of c\`adl\`ag paths from $[0,\infty)$ to $\bR$ equipped with Skorokhod's $J_1$-topology.
  
 The other one is the fluctuation scaling limit of the occupation time of a bilateral jumping-in diffusion $X$: $A(t) = \int_{0}^{t}1_{(0,\infty)}(X_s)ds$. In the positive recurrent case where $A(t)$ has a degenerate mean:
 \begin{align}
 	\frac{1}{t}A(t) \xrightarrow[t \to \infty]{P} p \in (0,1), \label{eq155} 
 \end{align}
 we establish the fluctuation scaling limits of the form:
 \begin{align}
  	g(\gamma)\left(\frac{A(\gamma t)}{\gamma} - p t\right) \xrightarrow[\gamma \to \infty]{f.d.} Z(t) \label{eq154}
 \end{align}
 for some divergent function $g$ and some limit process $Z(t)$. Here $\xrightarrow{f.d.}$ denotes the convergence of finite-dimensional distributions.
 
%
%

 \subsection{Previous study on scaling limits for diffusions}
 
 In order to study the inverse local times and occupation times for diffusions (without any jumps at all), we may assume without loss of generality that diffusions have the natural scale, whose local generators are of the form: $\frac{d}{dm}\frac{d^+}{dx}$ with a {\it speed measure} $m$. Here $\frac{d^+}{dx}$ denotes the right-differentiation operator. In the unilateral case, the speed measure $m$ comes from a string $m$, i.e.\ $m: (0,\infty) \to \bR$ is a non-decreasing, right-continuous function: $m(a,b] = m(b) - m(a)$. 
 
 Kasahara and Watanabe\cite{KasaharaWatanabe:Brownianrepresentation} has introduced the {\it renormalized inverse local time} $T(m;t)$ at $0$ of a unilateral $\frac{d}{dm}\frac{d^+}{dx}$-diffusion for a string $m$ with $\int_{0+}m(x)^2dx < \infty$, i.e.\ $\int_{0}^{\delta}m(x)^2dx < \infty$ for some $\delta> 0$. The process $T(m;t)$ is a L\'evy process without negative jumps nor Gaussian part. When the boundary $0$ is regular, the process $T(m;t)$ coincides with the inverse local time up to a constant drift. If $m$ has infinite mass around $0$, the boundary $0$ for $X$ is not regular; its inverse local time does not exist. They showed scaling limits of $T(m;t)$ as follows.
 \begin{Thm}{\rm (Kasahara and Watanabe\cite[Theorem 3.3]{KasaharaWatanabe:Brownianrepresentation})} \label{Thm: Kasahara and Watanebe 1}
 	Let $m$ be a string such that $\int_{0+}m(x)^2dx < \infty$ and $\alpha \in (1,2)$.
 	Assume that $m(x, \infty)$ varies regularly at $\infty$ of index $1/\alpha - 1$.
 	Then we have
 	\begin{align}
 	g(\gamma)\left(\frac{T(m;\gamma t)}{\gamma} - m(\infty) t\right) \xrightarrow[\gamma \to \infty]{d} S^{(\alpha)}(t) \ \text{on}\ \bD, \label{}
 	\end{align}
 	where $g(\gamma) = 1/m(\gamma,\infty)$ and $S^{(\alpha)}$ is a spectrally positive strictly $\alpha$-stable process.
 \end{Thm}
 They applied this result to show the limit \eqref{eq154} for bilateral diffusions (without any jumps at all).
 \begin{Thm}{\rm (Kasahara and Watanabe\cite[Theorem 4.1]{KasaharaWatanabe:Brownianrepresentation})} \label{Thm: Kasahara and Watanabe 2}
	Let $m$ be a Radon measure on $\bR$ such that $m(\bR) < \infty$ and $\alpha \in (1,2)$.
	Assume the following:
	\begin{align}
		m(x,\infty) \sim c_+x^{1/\alpha - 1}K(x),\  m(-\infty,-x) \sim c_-x^{1/\alpha - 1}K(x) \ (x \to \infty), \label{}
	\end{align}
	for constants $c_\pm > 0$ and a slowly varying function $K$ at $\infty$.
	Then we have
	\begin{align}
	g(\gamma)\left(\frac{A(\gamma t)}{\gamma} - p  t\right) \xrightarrow[\gamma \to \infty]{f.d.} (1-p)c_+ S^{(\alpha)}(t) - pc_- \tilde{S}^{(\alpha)}(t), \label{}
	\end{align}
	where $g(\gamma) = \frac{m(\bR)}{\gamma^{1/\alpha} K(\gamma)}$, $p = \frac{m(0,\infty)}{m(-\infty,0)}$ and $S^{(\alpha)}$ and $\tilde{S}^{(\alpha)}$ are i.i.d.\ spectrally positive strictly $\alpha$-stable processes.
 \end{Thm}

 Their idea of reducing the proof to the fluctuation scaling limit of the inverse local time relies upon the fact that every diffusion is realized as a time-changed Brownian motion. 
 
%

\subsection{Previous study on scaling limits for jumping-in diffusions with big jumps}
 Yano\cite{Yano:Convergenceofexcursion} has studied scaling limits of jumping-in diffusions. By Feller\cite{Feller:Theparabolic} and It\^o\cite{Ito:PPP}, under the natural scale, a jumping-in diffusion can be characterized by its speed measure $m$ and its jumping-in measure $j$, both of which are Radon measures on $(0,\infty)$. The jumping-in diffusion is a strong Markov process on $[0,\infty)$ which behaves as a $\frac{d}{dm}\frac{d^+}{dx}$-diffusion during staying in $(0,\infty)$ and jumps from the origin to $(0,\infty)$ according to $j$. We denote the jumping-in diffusion by $X_{m,j}$ and its inverse local time by $\eta_{m,j}$. For the precise description, see Section \ref{section: contthmofinverselocaltime}. One of the main results in Yano\cite{Yano:Convergenceofexcursion} is the following: 
 \begin{Thm}{\rm (Yano\cite[Theorem 2.6]{Yano:Convergenceofexcursion})}\label{Thm:Yano1}
	Let $\alpha \in (1,\infty)$ and $0 < \beta < 1/\alpha$.
 	Assume the following:
 	\begin{enumerate}
 		\item For a slowly varying function $K$ at $\infty$,
 		\begin{align}
 			m(x,\infty) \sim (\alpha - 1)^{-1}x^{1/\alpha - 1}K(x) \ (x \to \infty). \label{(M)_alpha} 
 		\end{align}
 		\item $j(x,\infty) \sim x^{-\beta}L(x)$ as $x \to \infty$ for a slowly varying function $L$ at $\infty$. 
 	\end{enumerate}
 	Then (with some technical conditions) we have the following convergence on $\bD$ in $J_1$-topology:
 	\begin{align}
 	\frac{1}{\gamma}X_{m,j}(\gamma^{1/\alpha} K(\gamma) t) &\xrightarrow[\gamma \to \infty]{d} X_{m^{(\alpha)}, j^{(\beta)}}(t), \label{} \\
 	\frac{1}{\gamma^{1/\alpha} K(\gamma)}\eta_{m,j}\left(\frac{\gamma^{\beta} t}{L(\gamma)}\right) &\xrightarrow[\gamma \to \infty]{d} S^{(\alpha\beta)}( t), \label{sl2}
 	\end{align}
 	where $m^{(\alpha)}(x,\infty) \propto x^{1/\alpha - 1}$ and $j^{(\beta)}(dx) \propto \beta x^{- \beta - 1}dx$ so that $X_{m^{(\alpha)}, j^{(\beta)}}$ is a $\alpha \beta$-self-similar jumping-in diffusion and $S^{(\alpha\beta)}$ is an $\alpha\beta$-stable subordinator. By the symbol $\propto$, we mean that the both sides coincide up to a multiplicative constant.
 \end{Thm}
 
 Roughly speaking, this result says under the big jump condition (the assumption (ii)), the scaling limit of the process $X_{m,j}$ converges to a jumping-in diffusion. 
 If we replace the assumption (ii) by a small jump condition, i.e. 
 \begin{align}
 	\text{(ii)}' \ \int_{0}^{\infty}xj(dx) < \infty, \label{}
 \end{align}
 the scaling limit of $X_{m,j}$ does not exist. The reason is as follows: Under the condition (ii)', the scaling limit of $X_{m,j}$ must be a diffusion (without any jumps at all). This can be seen from Theorem 2.5 of Yano\cite{Yano:Convergenceofexcursion}. However, the origin for $\frac{d}{dm^{(\alpha)}}\frac{d^+}{dx}$ is an exit boundary for $\alpha >1$ so that the process cannot exist without jumping-in. There is still a possibility that the scaling limit of $\eta_{m,j}$ exists, and it is the problem which we tackle in the present paper.
 
%
 
 
 
 \subsection{Main results}
 
 Our aim is to establish the scaling limits for positive recurrent jumping-in diffusions with small jumps by focusing on their fluctuations.
 
 The following theorem gives the fluctuation scaling limit of the inverse local times of unilateral jumping-in diffusions. For the precise statement, see Theorem \ref{alpha1-2}.
 
 \begin{Thm}\label{main-informal-1}
 	Let $\alpha \in (1,2)$. 
 	Suppose the following hold:
 	\begin{enumerate}
 		\item $\int_{0+}m(x)^2dx < \infty$,
 		\item $m(x,\infty) \sim (\alpha - 1)^{-1}x^{1/\alpha - 1}K(x) \ (x \to \infty)$ for a slowly varying function $K$ at $\infty$,
 		\item $\kappa := \int_{0}^{\infty}xj(dx) < \infty$.
 	\end{enumerate}
 	Then (with some technical conditions) we have
 	\begin{align}
 	&g(\gamma)\left(\frac{\eta_{m,j}(\gamma t)}{\gamma} - b t\right) \xrightarrow[\gamma \to \infty]{d}S^{(\alpha)}(\kappa t) \ \text{on}\ \bD,
 	\end{align}
 	where $S^{(\alpha)}$ is a spectrally positive strictly $\alpha$-stable process whose Laplace exponent is given as 
 	\begin{align}
 		\bE[\mathrm{e}^{-\lambda S^{(\alpha)}(t)}] = \mathrm{e}^{-t \chi(\lambda)}, \ \chi(\lambda) = -\frac{\Gamma(2-\alpha)}{\Gamma(\alpha)} \frac{\alpha^{\alpha-1}}{\alpha-1} \lambda^{\alpha} \ ( \lambda > 0), \label{eq160}
 	\end{align}
 	 and
 	\begin{align}
 	b = \int_{0}^{\infty}j(dx)\int_{0}^{x}m(y,\infty)dy, \ g(\gamma) = \frac{1}{ \gamma^{1/\alpha}K(\gamma)}. \label{}
 	\end{align}
 \end{Thm}

 The following theorem gives the fluctuation scaling limit of the occupation times of bilateral jumping-in diffusions. For the precise statement, see Theorem \ref{flucofA alpha1-2}.
 
 \begin{Thm}\label{main-informal-2}
	Let $\alpha \in (1,2)$. Suppose the following hold:
 	\begin{enumerate}
 		\item $\int_{0+}m_{\pm}(x)^2dx < \infty$,
 		\item $m_\pm(x,\infty) \sim w_{\pm}(\alpha - 1)^{-1}x^{1/\alpha -1}K (x) \  (x \to \infty)$ for  constants $\alpha  \in (1,2)$ and $w_{\pm} > 0$ and a slowly varying function $K$ at $\infty$, respectively,
 		\item $\kappa_\pm:= \int_{0}^{\infty}xj_\pm(dx) < \infty$.
 	\end{enumerate}
 	Then (with some technical conditions) we have
 	\begin{align}
 	g(\gamma)\left(\frac{A(\gamma t)}{\gamma} - p t\right) \xrightarrow[\gamma \to \infty ]{f.d.} (1-p)w_+S^{(\alpha)}(\tilde{\kappa}_+t) - pw_-\tilde{S}(\tilde{\kappa}_- t), \label{}
 	\end{align}
 	where 
 	\begin{align}
 	b_\pm &= 
 	\int_{0}^{\infty}j_\pm(dx)\int_{0}^{x}m_\pm(y,\infty)dy,\quad p = \frac{b_+}{b_+ + b_-}, \quad g(\gamma) = \frac{\gamma}{\gamma^{1/\alpha}K(\gamma)}, \label{} \\
 	 \tilde{\kappa}_\pm &= \frac{\kappa_\pm}{b_+ + b_-} \label{}
 	\end{align}
 	and, $S^{(\alpha)}$ and $\tilde{S}^{(\alpha)}$ are i.i.d. $\alpha$-stable processes characterized by \eqref{eq160}.
 \end{Thm}

 We also show Theorem \ref{main-informal-1} and \ref{main-informal-2} for $\alpha = 2,1$. See Section \ref{section: alpha2}, \ref{section: alpha1}, respectively.

 \subsection{Our strategy for the proofs}
 
%

 For one-dimensional diffusions (without negative jumps at all), Kasahara and Watanabe\cite{KasaharaWatanabe:Brownianrepresentation} proved Theorems \ref{Thm: Kasahara and Watanebe 1} and \ref{Thm: Kasahara and Watanabe 2} by utilizing the Brownian time-change method in order to apply It\^o formula. But we do not adopt the time-change method. 
 We explain our strategy to show Theorems \ref{main-informal-1} and \ref{main-informal-2}.
   
 
 We show Theorem \ref{main-informal-2} as an application of Theorem \ref{main-informal-1}.
 The key to the proof is to establish the tail behavior of the L\'evy measure of the inverse local time, which is obtained by applying Tauberian theorems to Theorem \ref{main-informal-1}.
 We then appeal to the It\^o excursion theory which connects the occupation time with the inverse local time.
  
 Our proof of Theorem \ref{main-informal-1} is done by proving a kind of continuity of the Laplace exponent of inverse local times. 
 To characterize the Laplace exponent of $\eta_{m,j}$, we introduce the eigenfunctions subject to the modified Neumann boundary condition.
 
 When the boundary $0$ for $dm$ is regular, we have a unique solution $u = \varphi_m(\lambda;\cdot)$ to
 \begin{align}
 	\frac{d}{dm}\frac{d^+}{dx}u = \lambda u, \  u(0) = 1, \  u^+(0) = 0, \label{}
 \end{align}
 and a unique solution $u = \psi_m(\lambda;\cdot)$ to
 \begin{align}
 	\frac{d}{dm}\frac{d^+}{dx}u = \lambda u, \  u(0) = 0, \  u^+(0) = 1. \label{}
 \end{align}
 We can exploit the two functions $\varphi_m$ and $\psi_m$ to analyze the $\frac{d}{dm}\frac{d^+}{dx}$-diffusion.
 When the boundary $0$ for $dm$ is exit, we still have $\psi_{m}$ but do not $\varphi_m$. We would like to introduce a function which plays the role of $\varphi_m$. We assume $\int_{0+}m(x)^2dx < \infty$, and then we can prove existence of a unique solution $u = \varphi^1_m$ to
 \begin{align}
 	\frac{d}{dm}\frac{d^+}{dx}u = \lambda u, \  u(0) = 1, \ \lim_{x \to +0}(u^+(x) - \lambda (m(x) - m(1))) = 0. \label{}
 \end{align} 
 We utilize the function $\varphi^1_m$ as follows.
 It is shown in Section \ref{section: contthmofinverselocaltime} that the Laplace exponent $\chi_{m,j}$ of the inverse local time $\eta_{m,j}$ is represented as follows:
 \begin{align}
 \chi_{m,j}(\lambda) = \int_{0}^{\infty}( 1 - g_m(\lambda;x) )j(dx). \label{}
 \end{align} 
 Here $u = g_m(\lambda;\cdot)$ is the unique non-increasing solution to 
 \begin{align}
 	\frac{d}{dm}\frac{d^+}{dx}u = \lambda u, \  u(0) = 1, \  \lim_{x \to \infty}u^+(x) = 0. \label{}
 \end{align}
 Then for a suitable constant $c^1_m(\lambda)$, we have the following expression:
 \begin{align}
 	g_m(\lambda;x) = \varphi^1_m(\lambda;x) - c^1_m(\lambda) \psi_m(\lambda;x). \label{eq159}
 \end{align}
 Therefore for a sequence of speed measures $\{dm_n\}_n$, jumping-in measures $\{j_n\}_n$ and constants $\{b_n\}_n$, the Laplace exponent $\tilde{\chi}_{m_n,j_n,b_n}$ of the process $\eta_{m_n,j_n}(t) - b_n t$ is the following:
 \begin{align}
 	\tilde{\chi}_{m_n,j_n,b_n}(\lambda) &= \chi_{m_n,j_n}(\lambda) - b_n\lambda \label{} \\
 	&= \left(\int_{0}^{\infty}(1 - \varphi^1_{m_n}(\lambda;x))j_n(dx) - b_n\lambda\right)
 	+ c^1_{m_n}(\lambda)\int_{0}^{\infty}\psi_{m_n}(\lambda;x)j_n(dx). \label{eq91}
 \end{align}
 Hence our study is reduced to the proper choice of $\{b_n\}_n$ and the analysis of the two terms in RHS of \eqref{eq91}. We show that the convergence of these terms follows from the convergence of $\{dm_n\}_n$ and $\{j_n\}_n$ in a certain sense.
 
 In our argument, the results in Kotani\cite{Kotani:Krein'sstrings} play a crucial role.
 He showed that strings $m$ satisfying $\int_{0+}m(x)^2dx < \infty$ have one-to-one correspondence to a class of Herglotz functions, and he also showed the correspondence is bi-continuous in some sense. See Section \ref{section: preliminary} for the detail.
 These results are an extension of the Krein correspondence which has been used in the studies of one-dimensional diffusions (see e.g. Kotani and Watanabe\cite{KotaniWatanabe:Krein'sspectraltheory} or Kasahara\cite{Kasahara:Spectraltheory}).
 Applying his result, we can obtain the explicit form of the coefficient $c^1_{m}(\lambda)$ in \eqref{eq159} and its continuous dependence on $m$.
  
 Let us remark on the non-degenerate case, that is, $\frac{1}{t}A(t)$ converges in law to a non-degenerate distribution. This case falls down to Lamperti's generalized arcsine law, which has been thoroughly studied in Watanabe\cite{Watanabe:Arcsinelaw} in the context of one-dimensional diffusions. He gave necessary and sufficient conditions for the convergence by the methods of double Laplace transforms and Williams formula via the excursion theory. His methods are still valid in our situation. See Appendix \ref{appendix: non-degenerate case}.

 \subsection{Outline of the paper}
 The remainder of the present paper is organized as follows: In Section \ref{section: preliminary}, we briefly review some results on Feller's classification of boundary and the Krein-Kotani correspondence. In Section \ref{section: constofphi}, we construct the functions $\varphi^1_m$ which play a role corresponding to $\varphi_m$ when the boundary $0$ is exit and establish some elementary estimates for them. In Section \ref{section: represenofg_lambda}, we represent $g_m$ as a linear combination of $\varphi^1_m$ and $\psi_m$ and, determine the coefficient. In Section \ref{section: contthmofinverselocaltime}, we show a continuity of inverse local times with respect to their speed measures and jumping-in measures. In Section \ref{section: scalinglimit}, we study the fluctuation scaling limit of inverse local times of jumping-in diffusion for scale parameter $\alpha \in (1,2)$.
 In Section \ref{section: arcsinelaw}, we discuss the fluctuation of the occupation time of bilateral jumping-in diffusions. In Section \ref{section: alpha2} and \ref{section: alpha1}, we treat the case scale parameter $\alpha = 2,1$, respectively. 
 In appendix \ref{appendix: contthmofLT}, we show a continuity theorem for Laplace transforms of spectrally positive  L\'evy processes.
 In appendix \ref{appendix: non-degenerate case}, we treat the case $\frac{1}{t}A(t)$ converges to non-degenerate distribution.
  \\
 
{\bf Acknowledgements} 
We would like to thank Shinichi Kotani, who read an early draft of this paper and gave us valuable comments. The research of Kouji Yano was supported by JSPS KAKENHI Grant Number
JP19H01791, JP19K21834 and JP18K03441.
 
 \section{The Krein-Kotani correspondence}\label{section: preliminary}
 	We consider the state interval $(a,\infty)$ for $a = 0$ or $-\infty$.
 	Let $dm$ be a Radon measure on $(a, \infty)$ with full support and let $s$ be a strictly increasing continuous function $(a,\infty)$.
 	We define 
 	\begin{align}
 		I = \int_{a}^{1}ds(y)\int_{a}^{y}dm(z), \ J = \int_{a}^{1}dm(y)\int_{a}^{y}ds(z). \label{}
 	\end{align}
 	Feller's classification of the boundary $a$ is as follows:
 	\begin{align}
 		\text{
 			\begin{tabular}{c||c|c}
 			& $I < \infty$ & $I = \infty$ \\ \hline \hline
 			$J < \infty$ & regular & exit \\ \hline
 			$J = \infty$ & entrance & natural 			
 			\end{tabular}
 		}
 		\end{align}
 	It is well-known that for such $m$ and $s$, there is a diffusion $X$ on $(a,\infty)$ or $[a, \infty)$ whose local generator on $(a,\infty)$ is given as $\frac{d}{dm}\frac{d}{ds}$. We call $m$ the speed measure of $X$ and $s$ the scale function of $X$.
 	It is also well-known that if the boundary $a$ is regular or exit, we have such a diffusion $X$ and $X$ hits or is killed at $a$ in finite time. When the boundary $a$ is regular, the boundary $a$ can be reflecting, absorbing or elastic and, when $a$ is exit, the boundary $a$ for $X$ is necessarily absorbing. When the boundary $a$ is natural or entrance, such a diffusion $X$ also exists and $X$ does not approach $a$ in finite time. For a diffusion $X$ with the generator $\frac{d}{dm}\frac{d}{ds}$,
 	the process $s(X)$ is also a diffusion and its generator is represented as $\frac{d}{d\tilde{m}}\frac{d^+}{dx}$, where $\tilde{m}(x)= m( s^{-1}(x))$. We say that $s(X)$ is the diffusion $X$ under the natural scale. We may always assume natural scale without loss of generality. 
 	
 	 				
 	Let us consider the state interval $(-\infty,\infty)$.
	We briefly summarize some results on the Krein-Kotani correspondence.
	See Kotani\cite{Kotani:Krein'sstrings} for the details.
	A function $w: (-\infty,\infty) \to [0,\infty]$ is called a {\it string} on $(-\infty, \infty)$ when $w$ is non-decreasing and right-continuous.
	We denote as $\cM_{\mathrm{circ}}$ the set of strings $w: (-\infty,\infty) \to [0,\infty]$ satisfying
	\begin{align}
		\int_{-\infty}^{b}x^2dw(x) < \infty \label{}
	\end{align}
	for some $b \in \bR$.
	For an element $w \in \cM_{\mathrm{circ}}$ and $\lambda > 0$, we consider the solution $u = f_w(\lambda;\cdot)$ to the following ODE:
	\begin{align}
		\frac{d}{dw}\frac{d^+}{dx}u = \lambda u, \ u(-\infty) = 1, \ u^+(-\infty) = 0 \ (x < \ell). \label{}
	\end{align}
	Here $\ell = \inf \{ x \in \bR \mid w(x) = \infty \}$.
		Then define
		\begin{align}
			h_w(\lambda) = b + \int_{-\infty}^{b}\left( \frac{1}{f_w(\lambda;x)^2} - 1 \right)dx + \int_{b}^{\ell}\frac{dx}{f_w(\lambda;x)^2} \ (\lambda > 0). \label{}
		\end{align}
		for some $b \in \bR$. Note that $h_w(\lambda)$ is finite for every $\lambda > 0$ and the function $h_w$ does not depend on the choice of $b$. Moreover, if we define $\tilde{h}_w(-\lambda) = h_w(\lambda) \ (\lambda > 0)$, then the function $\tilde{h}_w$ is a Herglotz function, that is, $\tilde{h}_w$ can be extended to the holomorphic function on the upper half plane and, it maps the upper half plane to itself. Hence from the general theory of Herglotz functions, for a constant $\alpha \in \bR$ and a Radon measure $\sigma$ on $[0,\infty)$ with $\int_{0}^{\infty}\frac{\sigma(d\xi)}{\xi^2 + 1} < \infty$, we have the following expression:
		\begin{align}
			h_w(\lambda ) = \alpha + \int_{0-}^{\infty}\left( \frac{1}{\xi + \lambda} - \frac{\xi}{\xi^2 + 1} \right)\sigma(d\xi). \label{eq103}
		\end{align}
		We note that the measure $\sigma$ in RHS of \eqref{eq103} is the spectral measure of the differential operator $-\frac{d}{dw}\frac{d^+}{dx}$. Hence we call $h(w;\cdot)$ the spectral characteristic function of $w$.
		
		Let $\cH$ be the set of functions which are expressed in the form of RHS of \eqref{eq103} for a constant $\alpha \in \bR$ and a Radon measure $\sigma$ on $[0,\infty)$ with $\int_{0}^{\infty}\frac{\sigma(d\xi)}{\xi^2 + 1} < \infty$.
		It was proved in Kotani\cite{Kotani:Krein'sstrings} that the map 
		$\cM_{\mathrm{circ}} \ni w \mapsto h_w \in \cH$ is bijective.
		We call this correspondence the Krein-Kotani correspondence.
		
		Let us consider the state interval $(0,\infty)$. A function $m: (0,\infty) \to (-\infty, \infty)$ is called a {\it string} on $(0,\infty)$ when $m$ is non-decreasing and right-continuous.
		We introduce the set of strings $\cM_1$ which we mainly treat in the present paper. A string $m$ is an element of $\cM_1$ when it is strictly increasing, right-continuous and satisfies $\int_{0+}m(x)^2dx < \infty$.
		Note that the condition $\int_{0+}m(x)^2dx < \infty$ implies the boundary $0$ for a $\frac{d}{dm}\frac{d^+}{dx}$-diffusion is exit or regular.
		For a string $m$ on $(0,\infty)$, we define
		\begin{align}
		m^{\ast}(x) = \inf \{ y > 0 \mid m(y) > x \} \ (x \in \bR). \label{}
		\end{align}
		Then $w(x) = m^{\ast}(x)$ is a string on $(-\infty,\infty)$. We call $m^{\ast}$ the dual string of $m$.
		
		We have the following fact:
		\begin{align}
			\text{For $m \in \cM_1$, its dual string $m^{\ast}$ is an element of $\cM_{\mathrm{circ}}$.}
		\end{align}
		In fact, from an elementary computation, it is easily checked that
		\begin{align}
		\int_{0+}m(x)^2dx < \infty \Leftrightarrow \int_{-\infty}x^2dm^{\ast}(x) < \infty. \label{}
		\end{align}
		For $m \in \cM_1$, we define
		\begin{align}
		H_m(\lambda) = h_{m^{\ast}}(\lambda) \ (\lambda  > 0). \label{}
		\end{align}
		An important consequence of the Krein-Kotani correspondence is the following theorem shown in Kasahara and Watanabe\cite{KasaharaWatanabe:Brownianrepresentation} which asserts a kind of continuity of the Krein-Kotani correspondence.
		\begin{Thm}{{\rm (Kasahara and Watanabe \cite[Theorem 2.9]{KasaharaWatanabe:Brownianrepresentation})}}\label{convequiv}\\
			Let $m_n, m \in \cM_1$ and $\sigma \geq 0$. Assume the following holds:
			\begin{enumerate}
				\item $\lim_{n \to \infty}m_n(x) = m(x)$ for every continuity point $x$ of $m$,
				\item $\lim_{x \to +0} \limsup_{n \to \infty} |\int_{0}^{x}m_n(y)^2 dy - \sigma^2| = 0$.
			\end{enumerate}
			Then we have for every $\lambda > 0$
			\begin{align}
			\lim_{n \to \infty}H_{{m_n}}(\lambda) = H_{m}(\lambda) - \sigma^2\lambda. \label{}
			\end{align}
		\end{Thm} 	

 \section{Construction of eigenfunctions of the generator}\label{section: constofphi}
 
 In this section, we introduce the function $\varphi^{1}_m$ which is the $\lambda$-eigenfunction of the generator $\frac{d}{dm}\frac{d^+}{dx}$ and satisfies the condition which we call modified Neumann boundary condition.
 
 We prepare some notation.
 \begin{Def}
 	For $m \in  \cM_1$, we define as follows:
 	\begin{align}
 		G_m(x) &= \int_{0}^{x}m(y)dy \ (x \geq 0) , \label{} \\
 		\tilde{m}(x) &= m(x) - m(1)  \ (x > 0) , \label{} \\
 		G^1_m(x) &= \int_{0}^{x}\tilde{m}(y)dy \ (x \geq 0). \label{}
 	\end{align}
 \end{Def}
 \begin{Def}
 		Let $U$ be a function of bounded variation on $(0,\infty)$ and $f$ be a function such that $\int_{0}^{x}|f||dU| < \infty$ for every $x > 0$. Here $|dU|$ is the total variation measure of the Stieltjes measure $dU$. 
 		We define
 		\begin{align}
 			U\bullet f (x) = \int_{0}^{x}f(y)dU(y). \label{}
 		\end{align}
 \end{Def}
 \begin{Def}
 	For $m \in \cM_1$, we define 
 	\begin{align}
 	s(x) &= x \   (x \geq 0), \label{} \\
 	\psi_m(\lambda; x) &= \sum_{k=0}^{\infty}\lambda^k ((s\bullet m \bullet)^k s)(x) \ (\lambda \in \bR,\   x \geq 0), \label{eq115} \\
 	g_m(\lambda;x) &= \psi_m(\lambda;x)\int_{x}^{\infty}\frac{dy}{\psi_m(\lambda;y)^2} \  (\lambda \in \bR,\  x \geq 0) \label{eq8}
 	\end{align}
 	where $(s\bullet m \bullet)^1 s = s \bullet m \bullet s$, $(s\bullet m \bullet)^2 s = s \bullet m \bullet s \bullet m \bullet s$, etc.
 \end{Def}
 The convergence in the RHS of \eqref{eq115} and \eqref{eq8} follows from the following proposition.
 \begin{Prop}\label{estofpsi}
 	For $m \in \cM_1$, the following hold for every $x \geq 0$ and $d \geq 0$:
 	\begin{align}
 	(s\bullet m \bullet)^d s(x) &\leq xE^d_m(0;x), \label{eq120} \\
 	m \bullet (s \bullet m \bullet )^{d-1} s(x) &\leq E^d_m(0;x), \label{eq121} \\
 	0 \leq \psi_m(\lambda;x) - \sum_{k=0}^{d-1}\lambda^k(s\bullet m \bullet)^k s(x) &\leq x |\lambda|^d E^d_m(|\lambda|;x), \label{eq122} \\
 	0 \leq \psi^+_{\lambda}(m;x) - 1 - \sum_{k=0}^{d-1}\lambda^{k+1} m \bullet (s \bullet m \bullet)^k s(x) &\leq |\lambda|^{d+1} E^{d}_m(|\lambda|;x) \label{eq123}
 	\end{align} 
 	where $(m \bullet s)^d(x) = \left( \int_{0}^{x}y dm(y) \right)^d$ and $E^d_m(\lambda;x) = (1/d!)(m \bullet s)^d (x) \mathrm{e}^{\lambda (m \bullet s)(x)}$.
 \end{Prop}
 \begin{proof}
 	First we show \eqref{eq121} by induction.
 	The case $d = 0$ is obvious.
 	Assume \eqref{eq121} holds for $k \geq 0$.
 	Then it holds that
 	\begin{align}
 		m \bullet (s \bullet m \bullet )^{k} s(x) 
 		&= m \bullet s \bullet (m \bullet (s \bullet m \bullet)^{k - 1}s)(x) \label{} \\
 		&\leq  m \bullet s \bullet E^k_m(0;x) \label{} \\
 		&\leq \int_{0}^{x}E^k_m(0;y)  ydm(y) \label{} \\
 		&= E^{k+1}_m(0;x), \label{}
 	\end{align}
 	which shows \eqref{eq121} for $d = k+1$. By induction we obtain \eqref{eq121}. From \eqref{eq121}, we have
 	\begin{align}
 		(s \bullet m \bullet)^d s(x) &= s \bullet m \bullet (s \bullet m \bullet)^{d-1}s(x) \label{} \\
 		&\leq \int_{0}^{x} E^d_m(0;y)dy \label{} \\
 		&\leq x E^d_m(0;x) \label{}
 	\end{align}
 	and, we obtain \eqref{eq120}.
 	Next we show \eqref{eq122}. Since it holds that
 	\begin{align}
 		\psi_m(\lambda;x) - \sum_{k=0}^{d-1}\lambda^k(s\bullet m \bullet)^k s(x)
 		=& \sum_{k=d}^{\infty}\lambda^k(s \bullet m \bullet)^k s(x) \label{} \\
 		=& \sum_{k=0}^{\infty}\lambda^{k+d}s\bullet (m\bullet s \bullet)^{k}(m \bullet(s \bullet m\bullet)^{d-1} s)(x) \label{} \\
 		\leq& |\lambda|^d m \bullet (s \bullet m \bullet)^{d-1}s(x) \sum_{k=0}^{\infty}|\lambda|^k(s \bullet m \bullet)^{k}s(x), \label{eq129}
 	\end{align}
	we have from \eqref{eq120} and \eqref{eq121},
 	\begin{align}
 		\eqref{eq129} 
 		\leq |\lambda|^d E^d_m(0;x) \sum_{k=0}^{\infty} \frac{x|\lambda|^k (m \bullet s)^k(x)}{k!} = x|\lambda|^d E^d_m(|\lambda|;x). \label{}
 	\end{align}
 	The proof of \eqref{eq123} is similar and so we omit it.
 \end{proof}
  \begin{Rem}\label{charofpsi}
 	\begin{enumerate}
 		\item The function $u = \psi_m(\lambda;\cdot)$ is the unique solution of the integral equation:
 		\begin{align}
 		u(x) = x + \lambda \int_{0}^{x}(x-y)u(y)dm(y) \ \text{for}\  \lambda \in \bR \ \text{and}\  x \in [0,\infty). \label{}
 		\end{align}
 		In other words, the function $u = \psi_{m}(\lambda; \cdot)$ is the unique solution of the ODE $\frac{d}{dm}\frac{d^+}{dx}u = \lambda u$ satisfying the boundary condition $u(0) = 0$ (Dirichlet) and $u^+(0) = 1$
 		\item The function $u = g_m(\lambda;\cdot)$ is the unique, non-negative and non-increasing solution of the equation $\frac{d}{dm}\frac{d^+}{dx}u = \lambda u$ satisfying the boundary condition $u(0+) = 1$ and $\lim_{x \to \infty}\frac{d^+}{dx}u(x) = 0$.
 		In fact, since we have
 		\begin{align}
 			g_m \psi^+_{m} - g^+_m\psi_m
 			&= g_m\psi^+_{m} - \psi^+_{m}\left(\int_{x}^{\infty}\frac{dy}{\psi_{m}(y)^2}\right)\psi_{m} + 1 = 1, \label{}
 		\end{align}
 		it follows that
 		\begin{align}
 			0 = d(g_m \psi^+_{m} - g^+_m\psi_m) 
 			= g_m d\psi^+_{m} - \psi_{m}dg^+_m 
 			= \lambda g_m \psi_{m}dm - \psi_{m}dg^+_m. \label{} 			
 		\end{align}
		.
 		Hence we obtain
 		\begin{align}
 			\frac{d}{dm}\frac{d^+}{dx}g_m = \lambda g_m \label{}
 		\end{align}
 		(this argument is due to It\^o\cite{Ito:Essentialsof}).
 	\end{enumerate}
 \end{Rem}
 Here we introduce the $\lambda$-eigenfunction announced in the beginning of this section.
 \begin{Def}
	For $m \in \cM_1$ and $x \geq 0$, we define
	\begin{align}
		\varphi^1_m(\lambda;x) &= 1 + \sum_{k=0}^{\infty}\lambda^{k+1}(s\bullet m \bullet)^k G^{1}_m(x). \label{eq9}
	\end{align}  
 \end{Def}
 The convergence of the summation in RHS of \eqref{eq9} follows from the following proposition.
   \begin{Prop}\label{estofphi^d}
 	Let $m \in \cM_1$. Then for any $d \geq 0$, the following hold for every $x \geq 0$:
 	\begin{align}
 	&\left|\varphi^1_m(\lambda;x) - 1 - \sum_{k=0}^{d}\lambda^{k+1} (s\bullet m \bullet)^{k}G^{1}_m(x)\right| \leq 
 	|\lambda|^{d+2} x \cdot S_m(x)\cdot E^d_m(|\lambda|;x), \label{eq124} \\
 	&\left|(\varphi^1_m)^+(\lambda;x) - \lambda \tilde{m}(x) - \sum_{k=1}^{d-1} \lambda^{k+1} m \bullet (s \bullet m\bullet)^{k-1}G^1_m(x) \right| \leq 
 	|\lambda|^{d+1}\cdot S_m(x) \cdot E^d_m(|\lambda|;x) \label{eq125}
 	\end{align} 
 	where $S_m(x) = \sup_{y \in [0,x]}|m \bullet G^1_m(y)|$.
 \end{Prop}
 \begin{proof}
	We only show \eqref{eq124}; the proof of \eqref{eq125} is similar.
 	From the definition of $\varphi^1_m$ and Proposition \ref{estofpsi}, for $x \geq 0$, we have
 	\begin{align}
 		\left|\varphi^1_m(\lambda;x) - 1 - \sum_{k=0}^{d}\lambda^{k+1} (s\bullet m \bullet)^{k}G^{1}_m(x)\right| 
 		=& \left| \sum_{k=d+1}^{\infty}\lambda^{k+1} (s \bullet m \bullet)^k G^1_m(x) \right| \label{} \\
 		\leq& S_m(x) |\lambda|^2 \sum_{k=d}^{\infty}\lambda^{k}(s \bullet m \bullet)^{k}s(x) \label{} \\
 		\leq& |\lambda|^{d+2} x\cdot S_m(x)\cdot E^d_m(|\lambda|;x). \label{} 
 	\end{align}
 \end{proof}

 The following theorem gives a characterization of the function $\varphi^1_m$ in terms of the integral and differential equations. 
 \begin{Thm} \label{rem1}
 	For $m \in \cM_1$, the following hold: 
	\begin{enumerate}
		\item The summation in \eqref{eq9} converges uniformly on every compact subset of $[0,\infty)$.
		\item \label{integralequation} The function $u = \varphi^1_m(\lambda;\cdot) - 1$ is the unique solution of the integral equation
		\begin{align}
		u(x) = \lambda G^{1}_m(x) + \lambda\int_{0}^{x}(x-y)u(y)dm(y). \label{eq36}
		\end{align}
		Equivalently, the function $u=\varphi^1_m(\lambda;\cdot)$ is the unique solution of the equation $\frac{d}{dm}\frac{d^+}{dx}u = \lambda u$ with the boundary condition:
		\begin{align}
		u(0) = 1, \lim_{x \to +0} \left(u^+(x) - \lambda \tilde{m}(x) \right) = 0. \label{}
		\end{align}
		We call this boundary condition the \emph{modified Neumann boundary condition}.
	\end{enumerate}
 \end{Thm}
 \begin{proof}
	The assertion (i) is obvious from Proposition \ref{estofphi^d} and for the assertion (ii), it is easily checked that $u = \varphi^1_m(\lambda;\cdot) - 1$ is the solution of the integral equation \eqref{eq36}. We prove the uniqueness. Let $u$ and $v$ be the solution of the integral equation \eqref{eq36}. Define $w(x) = u(x)- v(x)$. Then it holds that
	\begin{align}
	w(x) = \lambda\int_{0}^{x}dy\int_{0}^{y}w(z)dm(z). \ (x \in [0,\infty) ) \label{}
	\end{align}
	Hence we have $w(x) \leq x\frac{|\lambda|^k (m \bullet s)^k(x)}{k!} \int_{0}^{x}w(y)dm(y)$ for every $x \geq 0$ and $k\geq 1$. Then it follows that $w(x)=0 \ (x \geq 0 )$.
\end{proof}
We introduce a subset of $\cM_1$.
\begin{Def}
Define
\begin{align}
\cM_0 &= \{ m \in \cM_1 \mid \lim_{x \to +0}m(x) > -\infty  \}. \label{}
\end{align}
\end{Def}
A string $m \in \cM_0$ with $m(0+) \geq 0$ is called a Krein's string and extensively used in the studies of one-dimensional diffusions (see e.g.\cite{Kasahara:Spectraltheory}, \cite{Kotani:Krein'sstrings}). In these studies, the unique solution $u = \varphi_m$ of the differential equation $\frac{d}{dm}\frac{d^+}{dx}u = \lambda u \ (\lambda \geq 0)$ with $u(0) = 1$ and $u^+(0) = \lambda m(0+)$ plays an important role. The following proposition gives a relation between $\varphi_m$ and $\varphi^1_m$.
 \begin{Prop} \label{repforM_0}
 	If $m \in \cM_0$, then the function $u = \varphi^1_m(\lambda;\cdot) + \lambda m(1)\psi_m(\lambda;\cdot)$ is the unique solution of the integral equation: 
 	\begin{align}
 	u(x) = 1 + \lambda \int_{0-}^{x}(x-y)u(y)dm(y), \label{eq22}
 	\end{align}
 	where we regard $dm\{ 0 \} = m(0+)$.
 	If, in addition, $m(0+) \geq 0$, then the function $u(x) = \varphi^1_m(\lambda;\cdot) + \lambda m(1)\psi_m(\lambda;\cdot)$ is the unique non-decreasing solution of the differential equation:
 	\begin{align}
 	\frac{d}{dm}\frac{d^+}{dx}u = \lambda u,\ u(0) = 1,\ u^+(0) = \lambda m(0+). \label{}
 	\end{align}
 \end{Prop}
\begin{proof}
	Setting $v = \varphi^1_m - 1$, we have from \eqref{eq36},
	\begin{align}
		&\lambda \int_{0}^{x}dy \int_{0-}^{y}(\varphi^1_m + \lambda m(1)\psi_{m})dm(z) \label{} \\
		=& \lambda\int_{0}^{x}dy \int_{0}^{y}(v + 1 + \lambda m(1) \psi_{m})dm(z) + \lambda m(0+)x \label{} \\
		=& v - \lambda G^1_m + \lambda \int_{0}^{x}(m(y) - m(0+))dy + \lambda m(1)(\psi_{m} - x) + \lambda m(0+)x \label{} \\
		=& v + \lambda m(1)\psi_{m} \label{}
	\end{align}
 \end{proof}
 Since $\varphi^1_m(\lambda;\cdot)$ and $\psi_m(\lambda;\cdot)$ are linearly independent solutions of the equation $\frac{d}{dm}\frac{d^+}{dx}u = \lambda u$, the function $g_m(\lambda;\cdot)$ can be represented as 
 \begin{align}
 g_m(\lambda;\cdot) = \varphi^1_m(\lambda;\cdot) - c^1_m(\lambda)\psi_m(\lambda;\cdot) \label{defofg}
 \end{align} 
 by a constant $c^1_m(\lambda)$. 

 \section{Representation of spectral characteristic functions}\label{section: represenofg_lambda}
 
 In this section, we show for $m \in \cM_1$ it holds that $c^1_m(\lambda) = \lambda H_{m}(\lambda) - \lambda m(1) \ (\lambda > 0)$. This result is well-known in the case the boundary $0$ is regular. Therefore this is an extension of the result to a class of exit boundaries. The essential tool is the Krein-Kotani correspondence. 
 
 We note the following well-known result on one-dimensional diffusion theory (see e.g.\ \cite{KotaniWatanabe:Krein'sspectraltheory} for the proof).
 \begin{Prop}\label{Kreinfact}
 	Let $m \in \cM_1$ with $m(0+) \geq 0$ and $u = \varphi_m(\lambda;\cdot)$ be the unique solution of the following differential equation:
 	\begin{align}
 		\frac{d}{dm}\frac{d^+}{dx}u = \lambda u,\ u(0) = 1, \ u^+(0) = \lambda m(0+). \label{eq93} 
 	\end{align}
 	Then the following holds:
 	\begin{align}
 		g_m(\lambda;\cdot) = \varphi_m(\lambda;\cdot) - \lambda H_{m}(\lambda)\psi_m(\lambda;\cdot) \ (\lambda > 0). \label{}
 	\end{align}
 \end{Prop}
 The following proposition is obvious from Proposition \ref{repforM_0} and Proposition \ref{Kreinfact}.
 \begin{Prop} \label{Krein}
 	Let $m \in \cM_0$ and $\lambda > 0$. It holds that
 	\begin{align}
 		c^1_m(\lambda) = \lambda H_{m}(\lambda) - \lambda m(1). \label{}
 	\end{align}
 \end{Prop}
 We generalize Proposition \ref{Krein} by approximation argument.
 \begin{Thm} \label{charfunc}
 	Let $m \in \cM_1$ and $\lambda > 0$. It holds that
 	\begin{align}
 		c^1_m(\lambda) = \lambda H_{m}(\lambda) -\lambda m(1). \label{}
 	\end{align}
 \end{Thm}
 \begin{proof}
 	This proof is essentially due to that of Theorem 1 of \cite{Kasahara:Spectraltheory}.
 	We take a decreasing sequence $\{a_n \}_n$ which diverges to $-\infty$. Fix $\lambda > 0$. Define strings $\{m_n \}_n$ as $m_n(x) = \max\{m(x) , a_n\}$. We denote
 	\begin{align}
 		\varphi^1_{n}(x) = \varphi^1_{m_n}(x),\ \varphi^1(x) = \varphi^1_{m}(x),\ etc. \label{}
 	\end{align}
	Since $m_n(0+) > -\infty$, it holds that from Proposition \ref{Krein}
 	\begin{align}
 		c^1_n = \lambda H_n -\lambda m_n(1). \label{}
 	\end{align}
 	In order to prove 
 	\begin{align}
 		g_m = (\varphi^1_n + \lambda m(1)\psi_n ) - \lambda H_{m^{\ast}}\psi_n, \label{}
 	\end{align}
 	it is enough to show that as $n \to \infty$,
 	\begin{align}
 		H_n \to H, \quad \varphi^1_n(x) \to \varphi^1(x),\quad 
 		\psi_n(x) \to \psi(x), \quad
 		g_n(x) \to g(x) \label{}
 	\end{align}
 	 hold for some $x > 0$. The first one directly follows from Theorem \ref{convequiv}. At first, we prove $\lim_{n \to \infty}\psi_n(x) = \psi(x)$ for every $x > 0$. Since it holds that $(m_n \bullet s)(x) \leq (m \bullet s)(x)$, from Proposition \ref{estofpsi}, we have
 	\begin{align}
 		x \leq \psi_n(x) \leq x \mathrm{e}^{\lambda (m \bullet s)(x)}, \ 
 		1 \leq \psi^+_n(x) \leq \mathrm{e}^{\lambda (m \bullet s)(x)}. \label{eq7}
 	\end{align}
 	Then from Ascoli-Arzela theorem and the diagonal argument, we can take a subsequence $\{n_k \}_k$ such that $\{ \psi_{n_k} \}_k$ converges to some $\tilde{\psi}$ uniformly on every compact subset of $[0,\infty)$. Then from Remark \ref{charofpsi}, the following holds: 
 	\begin{align}
 		\psi_{n_k}(x) = x + \lambda \int_{0}^{x}(x-y)\psi_{n_k}(y)dm_n(y). \label{}
 	\end{align}
 	Then by $n \to \infty$, we obtain
 	\begin{align}
 		\tilde{\psi}(x) = x + \lambda \int_{0}^{x}(x-y)\tilde{\psi}(y)dm(y). \label{}
 	\end{align}
	Hence by the uniqueness of the solution of the integral equation, we obtain $\tilde{\psi} = \psi$. This argument also holds if we start from any subsequence of $\{\psi_n \}_n$. Hence we have
	\begin{align}
		\lim_{n \to \infty}\psi_n(x) = \psi(x) \label{}
	\end{align}
	 for every $x \in [0,\infty)$ and $\lambda > 0$. Since we have \eqref{eq8} and \eqref{eq7}, by the dominated convergence theorem, we obtain 
	 \begin{align}
	 	\lim_{n \to \infty}g_n(x) = g(x) \label{}
	 \end{align}
	 for every $x > 0$ and $\lambda > 0$. Finally, we prove that $\lim_{n \to \infty}\varphi^1_n(x) = \varphi^1(x)$. From the definition of $\varphi^1_n$, we can easily check the following holds for $x \in (0,\infty)$: 
 	\begin{align}
 		|\varphi^1_n(x)| & \leq 1 + \lambda \int_{0}^{x}|m(y)|dy + \lambda^2 x \mathrm{e}^{\lambda (m \bullet s)(x)}\cdot\sup_{y \in [0,x]}|m\bullet G^1_m(y)| , \label{}\\
 		|(\varphi^1_n)^+(x)| & \leq \lambda|m(x)| + \lambda^2 \mathrm{e}^{\lambda (m \bullet s)(x)}\cdot\sup_{y \in [0,x]}|m\bullet G^1_m(y)|. \label{}
 	\end{align}
 	Then we can take a subsequence $\{n_k \}_k$ such that $\varphi^1_{n_k}$ converges to some function $\tilde{\varphi}$ uniformly on every compact subset of $(0,\infty)$. Since the function $u = \tilde{\varphi} - 1$ is a solution of the integral equation 
 	\begin{align}
 		u(x) = \lambda G^1_m(x) + \lambda \int_{0}^{x}(x-y)u(y)dm(y). \label{}
 	\end{align}
 	From the uniqueness of the solution, we obtain 
 	\begin{align}
 		\tilde{\varphi}(x) = \varphi^1(x). \label{}
 	\end{align}
 	Hence it holds that $\lim_{n \to \infty}\varphi^1_n(x) = \varphi^1(x)$.
 \end{proof}
 The following proposition shows that the second order at $0$ of $g_m$ is $x$.
 \begin{Prop} \label{diffcirc}
 	Let $m \in \cM_1$. Then the following holds: 
 	\begin{align}
 		\lim_{x \to 0}\frac{1-g_m (\lambda;x) +\lambda G_m(x)}{x} = \lambda H_{m}(\lambda) \  \text{for every} \ \lambda > 0 . \label{}
 	\end{align}
 \end{Prop}
 \begin{proof}
 	From Theorem \ref{charfunc} and Theorem \ref{rem1} \eqref{integralequation}, it holds that 
 	\begin{align}
 		&1-g_m +\lambda G_m \label{} \\
 		=& \lambda (H_{m} - \lambda m(1))\psi_m + \lambda m(1)x -\int_{0}^{x}dy \int_{0}^{y}(\varphi^1_m -1)dm(y). \label{}
 	\end{align}
 	 Since $\psi^+_m(0+) = 1$ holds, we obtain the desired result.
 \end{proof}

\section{Convergence of inverse local times}\label{section: contthmofinverselocaltime}

 We consider a strong Markov process on $[0,\infty)$ which has continuous paths on $(0,\infty)$. Under Feller's boundary condition, such processes are shown in It\^o and Mckean\cite{ItoMcKean} and It\^o\cite{Ito:PPP} to have the generator $\cL$ of the following form: 
 \begin{equation}
 	\begin{cases}
 	r\cL u(0) = \int_{0}^{\infty}(u(y) - u(0))j(dy) + cu^+(0) \label{}, \\
 	\cL u(x) = \frac{d}{dm}\frac{d^+}{dx}u(x) \ \text{for} \ x \in (0,\infty).  \label{}
 \end{cases} 	
 \end{equation} 
 
  Here $m$ and $j$ are Radon measures on $(0,\infty)$ and constants $c, r$ are non-negative. Hence such processes are characterized by the quadruplet $(m,j,c,r)$. A necessary and sufficient condition for the existence of such process for given $(m,j,c,r)$ is shown in Feller\cite{Feller:Theparabolic} and It\^o\cite{Ito:PPP}.
  The conditions are the following:
  \begin{equation}
  	\text{(C)'}
  	\left\{
  		\begin{aligned}
  			&\text{(i) For some}\  x_0 > 0\ \text{the following holds:}   \\
  			&\quad j(x_0, \infty) + \int_{0}^{x_0}xj(dx) + \int_{0}^{x_0}|G_m(x)|j(dx) < \infty,  \\
  			&(ii) \ c = 0 \ \text{if}\ m(0+) = -\infty,  \\
  			&(iii) \ r > 0 \ \text{if}\  c = 0 \ \text{and}\ j(0,x_0) < \infty.
  		\end{aligned}
  		\label{existenceoftheprocess}
  	\right.
  \end{equation}
 For the quadruplet $(m,j,c,r)$ which satisfies (C)', we denote the corresponding process as $X_{m,j,c,r}$.
 Our main concern here is in the case: $m(0+) = -\infty$, that is, when the origin is the non-entrance boundary in the sense of Feller. Hence we assume $c = 0$. In addition we assume $r = 0$ for simplicity. Eventually our process $X_{m,j,c,r}$ depends only on $m$ and $j$. Then we denote $X_{m,j,c,r}$ simply as $X_{m,j}$. In this case, the condition (C)' is reduced to the following condition:
 \begin{equation}
 	\text{(C)}
 	\left\{
 	\begin{aligned}
 		&\text{For some}\  x_0 > 0\ \text{the following holds:} \\
 		&(i)\  j(x_0, \infty) + \int_{0}^{x_0}xj(dx) + \int_{0}^{x_0}|G_m(x)|j(dx) < \infty,  \\
 		&(ii)\  j(0,x_0) = \infty.
 	\end{aligned}
 	\label{existenceoftheprocess2}
 	\right.
 \end{equation}
 We briefly summarize the construction of the process $X_{m,j}$.
 Let $m$ and $j$ be Radon measures on $(0,\infty)$. Assume $(m,j)$ satisfies the condition (C).
 Define $T_0(e) = \inf\{ s > 0 \mid e(s) = 0 \} \ (e \in \bD)$ and $\bE$ as the set of all elements $e$ in $\bD$ which satisfy that $e(u) = 0$ for every $u \geq T_0(e)$ if $T_0(e) < \infty$.
 Let $P^m_x \ (x \in (0,\infty))$ be the law of the diffusion process with speed measure $dm$ starting from $x$ and killed at $0$.
 Then the following measure $n_{m,j}$ is the excursion measure of the process $X_{m,j}$:
 \begin{align}
 	n_{m,j}(A) = \int_{0}^{\infty}P^{m}_x (A)j(dx) \ (A \in \cB(\bE)). \label{} 
 \end{align} 
 We construct the sample paths of $X_{m,j}$.
 We define $N(ds de)$ as a Poisson random measure on $(0,\infty)\times \bE$ having intensity measure $dx \otimes n_{m,j}$ being defined on some probability space. Here $dx$ is the Lebesgue measure. We define $D(p) = \{ s \in (0,\infty) \mid N(\{s\}\times \bE ) = 1 \}$ and the map $p:D(p) \to \bE$ 
 such that $p[s] \ (s \in D(p))$ is the only one element of the support of the measure $N(\{s\} \times de)$.
 We define the process $\eta_{m,j}$ as follows:
 \begin{align}
 	\eta_{m,j}(u) = \int_{(0,u] \times \bE}T_0(e) N(ds de). \label{}
 \end{align}
 Then we construct $X_{m,j}$ as follows: 
 \begin{equation}
 	X_{m,j}(t) = \left\{
 	\begin{aligned}
 		&p[u](t - \eta_{m,j}(u-)) & &\text{if} \ u \in D(p) \ \text{and} \ \eta_{m,j}(u-) \leq t < \eta_{m,j}(u),  \\
 		&0 & &\text{otherwise}.
 	\end{aligned}
 	\right. \label{}
 \end{equation}
 Then $\eta_{m,j}$ plays the role of the inverse local time at $0$ of $X_{m,j}$.
 This process $X_{m,j}$ thus constructed is the jumping-in diffusion associated to $m$ and $j$ and started from $0$. We denote the law of $X_{m,j}$ by $P$.
 We have for $\lambda > 0$,
 \begin{align}
 	\chi_{m,j}(\lambda) &= -\log P[\mathrm{e}^{-\lambda \eta_{m,j}(1)}] \label{} \\
 	&= \int_{0}^{\infty}(1 - \mathrm{e}^{-\lambda u}) n_{m,j}(T_0 \in du)  \label{} \\
 	&= \int_{0}^{\infty}P^m_x[ 1 - \mathrm{e}^{-\lambda T_0}]j(dx) \label{}
 \end{align} 
 It is well-known that the following holds (see e.g.\cite{Ito:Essentialsof}):
 \begin{align}
 	g_m(\lambda;x) = P^m_x[\mathrm{e}^{-\lambda T_0}]. \label{}
 \end{align}
 Hence we obtain the following expression:
 \begin{align}
 	\chi_{m,j}(\lambda) &= \int_{0}^{\infty}(1 - g_m(\lambda;x))j(dx)\ (\lambda > 0 ). \label{eq94}
 \end{align}
 Here we establish two different continuity theorems for $\eta_{m,j}$. 
 \begin{Thm}\label{convtojump} 
 	Let $m_n,m \in \cM_1$ and $j_n$ be a Radon measure on $(0,\infty)$ and assume $(m_n,j_n)$ satisfies $\mathrm{(C)}$. Suppose the following hold: 
 	\begin{enumerate}
 		\item $\lim_{n \to \infty} m_n(x) = m(x)$ for every continuity point $x$ of $m$,
 		\item $\lim_{x \to +0}\limsup_{n \to \infty}\int_{0}^{x}m_n(y)^2dy = 0$,
 		\item $j_n(dx) \xrightarrow[n \to \infty]{w} 0$ on $[1,\infty]$,
 		\item $xj_n(dx) \xrightarrow[n \to \infty]{w} \kappa \delta_0(dx)$ on $[0,1]$ for a constant $\kappa > 0$.
 	\end{enumerate} 
 	Then if we take $b_n = -\int_{0}^{1}G_{m_n}(x)j_n(dx)$, we have
 	\begin{align}
 		\eta_{m_n,j_n}(t) - b_n t \xrightarrow[n \to \infty]{d} T(m;\kappa t) \ \text{on} \ \bD. \label{}
 	\end{align}
 	Here $T(m;t)$ is the L\'evy process without negative jumps whose Laplace exponent is $\lambda H_m(\lambda)$.
 \end{Thm}
 \begin{proof}
	From \eqref{eq94} and Proposition \ref{continuityofLT}, it is enough to show the following for every $\lambda > 0$:
 	\begin{align}
 		\lim_{n \to \infty}\left( \int_{0}^{\infty}(1-g_{m_n}(\lambda;x))j_n(dx) + \lambda \int_{0}^{1}G_{m_n}(x)j_n(dx)\right) = \kappa\lambda H_m(\lambda). \label{eq116}
 	\end{align}
 	Fix $\lambda > 0$.
 	Since the function $1 - g_{m_n}$ is bounded continuous,	we see from the assumption (iii) that \eqref{eq116} is equivalent to the following (with $(\lambda)$ being omitted):
 	\begin{align}
 	\lim_{n \to \infty} \int_{0}^{1}(1-g_{m_n} + \lambda G_{m_n})dj_n = \kappa\lambda H_m. \label{eq128}
 	\end{align}
 	From \eqref{defofg}, we have
 	\begin{align}
 		g_{m_n} = \varphi^1_{m_n} - c^1_{m_n}\psi_{m_n} 
 		= (1 + \lambda G^1_{m_n} + \Phi^1_{m_n}) - c^1_{m_n}(x + \Psi_{m_n}) \label{}
 	\end{align}
 	where $\Phi^1_{m_n} = \varphi^1_{m_n} - 1 - \lambda G^1_{m_n}$, $\Psi_{m_n} = \psi_{m_n} - x$.
 	Note that from Schwarz's inequality, we have for $x \in (0,1]$,
 	\begin{align}
 	(m_n \bullet s)(x) &\leq x \tilde{m}_n(x) - \int_{0}^{x}\tilde{m}_n(y)dy
 	\leq \int_{0}^{x}(-\tilde{m}_n(y))dy \leq \sqrt{x}\left( \int_{0}^{x}\tilde{m}_n(y)^2dy \right)^{1/2}, \label{eq140}
 	\end{align}
 	and hence for a fixed constant $\delta \in (0,1)$, from Proposition \ref{estofphi^d}, we have for a constant $C > 0$,
 	\begin{align}
 		\int_{0}^{1}|\Phi^1_{m_n}|dj_n &\leq \lambda^2 \int_{0}^{\delta}xS_{m_n}(x)\mathrm{e}^{\lambda (m_n \bullet s)} j_n(dx) 
 		+ \lambda^2\int_{\delta}^{1}xS_{m_n}(x)\mathrm{e}^{\lambda (m_n \bullet s)} j_n(dx)\label{} \\
 		&\leq C\lambda^2 \left( \int_{0}^{\delta}\tilde{m}_n(y)^2dy \right)\int_{0}^{\delta}x j_n(dx) + C\lambda^2 \int_{\delta}^{1}x j_n(dx). \label{}
 	\end{align}
 	Then from the assumptions (ii), (iii) and (iv), it follows that
 	\begin{align}
 		\limsup_{n \to \infty}\int_{0}^{1}|\Phi^1_{m_n}|dj_n &\leq \kappa C\lambda^2 \limsup_{n \to \infty}\left( \int_{0}^{\delta}\tilde{m}_n(y)^2dy \right)
 		\xrightarrow[\delta \to +0]{} 0 \ \text{by (ii)}. \label{eq126}
 	\end{align}
 	Similarly, by Proposition \ref{estofpsi}, we can show
 	\begin{align}
 		\lim_{n \to \infty}\int_{0}^{1}\Psi_{m_n}dj_n = 0. \label{eq127}
 	\end{align}
 	Since we have
 	\begin{align}
 		1 - g_{m_n} + \lambda G_{m_n} &= 1 - ((1 + \lambda G^1_{m_n} + \Phi^1_{m_n}) - c^1_{m_n}(x + \Psi_{m_n})) + \lambda G_{m_n} \label{} \\
 		&= -\Phi^1_{m_n} + \lambda H_{m_n} x + c^1_{m_n}\Psi_{m_n}. \label{}
 	\end{align}
 	Note that from Theorems \ref{convequiv} and \ref{charfunc}, it holds that
 	\begin{align}
 		\lim_{n \to \infty}H_{m_n} = H_{m}, \ \sup_n c^1_{m_n} < \infty. \label{}
 	\end{align}
 	Hence from \eqref{eq126} and \eqref{eq127}, we obtain \eqref{eq128}.
 \end{proof}

 The following is the second continuity theorem, which we need to show the case the scaling limit of inverse local times converges to a Brownian motion, which is treated in Section \ref{section: alpha2}.
 \begin{Thm}\label{convtoBM}
 	Let $m_n, m \in \cM_1$, $j_n$ be a Radon measure on $(0,\infty)$ and $\sigma_1$ be a non-negative constant and assume $(m_n,j_n)$ satisfies $\mathrm{(C)}$. Suppose the following hold: 
 	\begin{enumerate}
 		\item $\lim_{n \to \infty} m_n(x) = m(x)$ for every  continuity point $x$ of $m$,
 		\item $\lim_{x \to +0}\limsup_{n \to \infty} |\int_{0}^{x}m_n(y)^2dy - \sigma_1^2| = 0$ for a constant $\sigma > 0$,
 		\item $j_n(dx) \xrightarrow[n \to \infty]{w} 0$ on $[1,\infty]$,
 		\item $xj_n(dx) \xrightarrow[n \to \infty]{w} \kappa_1\delta_0(dx)$ on $[0,1]$ for a constant $\kappa_1 > 0$,
 		\item $(-s\bullet m_n \bullet G_{m_n}(x))j_n(dx) \xrightarrow[n \to \infty]{w} \kappa_2 \sigma_1^2\delta_0(dx)$ on $[0,1]$ for a constant $\kappa_2 \in \bR$.
 	\end{enumerate} 
 	Then if we take $b_n = -\int_{0}^{1}G_{m_n}(x)j_n(dx)$, we have $\kappa_1 \geq \kappa_2$ and
 	\begin{align}
 	\eta_{m_n,j_n}(t) - b_n t \xrightarrow[n \to \infty]{d} \sigma B(t) + T(m;\kappa t) \ \text{on} \ \bD, \label{}
 	\end{align}
 	where 
 	\begin{align}
 		\sigma^2 = (\kappa_1 - \kappa_2)\sigma_1^2, \ \kappa = \kappa_1, \label{} 
 	\end{align}
 	and $B$ is a standard Brownian motion independent of $T(m;\cdot)$.
 \end{Thm}
 \begin{proof}
 	From the same argument in Theorem \ref{convtojump}, it is enough to show that $\kappa_1 \geq \kappa_2$ and, for every $\lambda > 0$,
 	\begin{align}
 	\lim_{n \to \infty}\int_{0}^{1}(1-g_{m_n}(\lambda;x)) + \lambda G_{m_n}(x))j_n(dx) = \lambda H_m(\lambda) - \sigma^2\lambda^2. \label{eq119}
 	\end{align}
 	Fix $\lambda > 0$. From \eqref{defofg}, we have
 	\begin{align}
 		g_{m_n} = \varphi^1_{m_n} - c^1_{m_n} \psi_{m_n} = (1 + \lambda G^1_{m_n} + \lambda^2 s \bullet {m_n} \bullet G^1_{m_n} + \Phi^2_{m_n}) - c^1_{m_n} (x + \Psi_{m_n}), \label{}
 	\end{align}
 	where
 	\begin{align}
 		\Phi^2_{m_n} = \varphi^1_{m_n} - 1 - \lambda G^1_{m_n} - \lambda^2 s \bullet {m_n} \bullet G^1_{m_n}. \label{}
 	\end{align}
 	Take $\delta \in (0,1)$. From Proposition \ref{estofphi^d}, the assumption (ii) and the similar argument in Theorem \ref{convtojump}, we have for a constant $C > 0$,
 	\begin{align}
 		\int_{0}^{1}|\Phi^2_{m_n}|dj_n \leq C\lambda^3 (m_n \bullet s)(\delta)\int_{0}^{\delta}xj_n(dx) + C\lambda^3\int_{\delta}^{1}xj_n(dx), \label{eq130}
 	\end{align}
 	and we obtain $\eqref{eq130} \to  0 \ (n \to \infty)$.
 	Since we have
 	\begin{align}
 		1 - g_{m_n} + \lambda G_{m_n} &= 1 - ((1 + \lambda G^1_{m_n} + \lambda^2 s \bullet {m_n} \bullet G^1_{m_n} + \Phi^2_{m_n}) - c^1_{m_n} (x + \Psi_{m_n})) + \lambda G_m \label{} \\ 		
 		&=- \lambda^2 s \bullet {m_n} \bullet G^1_{m_n} - \Phi^2_{m_n} + \lambda H_{m_n} x +c^1_{m_n} \Psi_{m_n}, \label{}
 	\end{align}
 	we obtain \eqref{eq119} from the assumption (v).
	Finally, we show $\kappa_1 \geq \kappa_2$.
	Since it holds that for $x \in [0,1]$,
	\begin{align}
	-s\bullet m \bullet G_{m_n}(x) &= -s\bullet m \bullet G^1_{m_n}(x) -m_n(1)\int_{0}^{x}(m_n \bullet s)(y)dy \label{} \\
	&\leq x  \int_{0}^{x}\tilde{m}_n(y)^2dy -xm_n(1)(m_n \bullet s)(x), \label{}
	\end{align}
	from the assumptions (iv) and (v), we obtain $\kappa_1 \geq \kappa_2$.
 \end{proof}

 \section{Scaling limit of inverse local times}\label{section: scalinglimit}
 Applying the continuity theorems shown in the previous section, we establish the scaling limit of $\eta_{m,j}$. 
 
 We define a string $m^{(\alpha)} \  (\alpha \in (0,2))$ which is the speed measure of a Bessel process of dimension $2 -2 \alpha$ under the natural scale. 
 \begin{Def}
 	For $\alpha \in (0,2)$, we define for $x > 0$
 	\begin{equation}
 	m^{(\alpha)}(x) = \left\{
 	\begin{aligned}
 	&(1 - \alpha)^{-1}x^{1/\alpha -1} & &\text{if} \ \alpha \in (0,1), \label{} \\
 	&\log x & &\text{if} \ \alpha = 1, \label{} \\
 	&-(\alpha - 1)^{-1}x^{1/\alpha -1} & &\text{if} \ \alpha \in (1,2). \label{}
 	\end{aligned}
 	\right. \label{eq131}
 	\end{equation}
 \end{Def}
 \begin{Rem}
 	For $\alpha \in (0,2)$, we have for $\lambda > 0$
 	\begin{align}
 	H_{m^{(\alpha)}}(\lambda) =
 	\begin{cases}
 	\frac{\Gamma(2-\alpha)}{\Gamma(\alpha)} \frac{\alpha^{\alpha-1}}{1-\alpha} \lambda^{\alpha-1}, & \text{if}\ \alpha \in (0,1), \\
 	- (\log \lambda + 2\gamma), & \text{if}\ \alpha =1, \\
 	-\frac{\Gamma(2-\alpha)}{\Gamma(\alpha)} \frac{\alpha^{\alpha-1}}{\alpha-1} \lambda^{\alpha-1}, & \text{if}\ \alpha \in (1,2), \\
 	\end{cases} \label{}
 	\end{align}
 	where $\gamma$ is Euler's constant. Therefore the process $T(m^{(\alpha)};t)$ is a $\alpha$-stable process without negative jumps. See \cite{Kotani:Krein'sstrings} for the details.
 \end{Rem}
 We reduce the scaling limit of $\eta_{m,j}$ to the continuity of it with respect to $m$ and $j$ by the change of variables.
 For every $m \in \cM_1$, $a,b > 0$ and $x > 0$, it holds that 
 	\begin{align}
 	g_{m}(a\lambda;bx) &= g_{am^b}(b\lambda;x), \label{} \\
 	\psi_{m}(a\lambda;bx) &= b\psi_{am^b}(b\lambda;x), \label{} \\
 	\varphi^1_{m}(a\lambda;bx) &= \varphi^1_{am^b}(b\lambda;x), \label{}
 	\end{align}
 	where $m^b(x) = m(bx)$.
 	
 In order to show the tali behavior of $\eta_{m,j}$, we utilize the Tauberian theorems for finite Radon measures which were used in \cite{Kasahara:Tailsof}.
 \begin{Thm}{\label{KasaharaTauberian}}{\rm (Kasahara\cite[Theorem 2.1]{Kasahara:Tailsof})}
 	Let $\mu$ be a finite Radon measure on $[0,\infty)$ and $K$ be a slowly varying function at $\infty$. Define 
 	\begin{align}
 	f(\lambda) = \int_{0-}^{\infty}\mathrm{e}^{-\lambda x}\mu(dx) \ (\lambda > 0). \label{}
 	\end{align}
 	Then for $\beta > 0$ and $n > \beta$, the following are equivalent:
 	\begin{enumerate}
 		\item $\mu [x,\infty) \sim x^{-\beta} K(x) \ (x \to \infty)$,
 		\item $(-1)^n \frac{d^n}{d\lambda^n}f(\lambda) \sim \beta \Gamma(n - \beta)\lambda^{\beta - n}K(1/\lambda) \ (\lambda \to +0)$.
 	\end{enumerate}
 \end{Thm}

 Now we show the scaling limit of $\eta_{m,j}$.
 \begin{Thm}\label{alpha1-2}
 	Let $m \in \cM_1$, $j$ be a Radon measure on $(0,\infty)$ and  $K$ be a slowly varying function at $\infty$ and, assume $(m,j)$ satisfies $\mathrm{(C)}$.
 	Suppose the following hold: 
 	\begin{enumerate}
 		\item $m(x,\infty) \sim (\alpha-1)^{-1}x^{1/\alpha - 1}K(x) \ (x \to \infty)$ for a constant $\alpha \in (1,2)$,
 		\item $\kappa := \int_{0}^{\infty}xj(dx) < \infty$.
 	\end{enumerate} 
 	Then we have
 	\begin{align}
 	&\frac{1}{\gamma^{1/\alpha}K(\gamma)}(\eta_{m,j}(\gamma t) - b\gamma t) \xrightarrow[\gamma \to \infty]{d}T(m^{(\alpha)};\kappa t) \ \text{on}\ \bD, \label{} \\
 	&n_{m,j}[T_0 > s] \sim \frac{\kappa\alpha^{\alpha-1}}{\Gamma(\alpha)} s^{-\alpha}L^{\sharp}(s)^{-\alpha} \ (s \to \infty), \label{eq110}
 	\end{align}
 	where $L^{\sharp}(x)$ be a de Bruijn conjugate of $L(x) = K(x^{\alpha})$ and $b = \int_{0}^{\infty}j(dx)\int_{0}^{x}m(y,\infty)dy$.
 \end{Thm}
 \begin{proof}
 	We may assume $m(\infty) = 0$ without loss of generality.
 	Define
 	\begin{align}
 		m_\gamma(x) = \frac{m(\gamma x)}{\gamma^{1/\alpha-1}K(\gamma)}, \ j_\gamma(dx) = \gamma j(d(\gamma x)), \tilde{b}_\gamma = -\int_{0}^{\infty}G_{m_\gamma}dj_\gamma. \label{}
 	\end{align}
	Then
 	\begin{align}
 		\frac{1}{\gamma^{1/\alpha}K(\gamma)}(\eta_{m,j}(\gamma t) - b\gamma t)
 		 \overset{d}{=} \eta_{m_\gamma,j_\gamma}(t) - \tilde{b}_\gamma t \ \text{on} \ \bD. \label{}
 	\end{align}
	Then it is enough to show that $\{ m_\gamma \}_\gamma$ and $\{ j_\gamma \}_\gamma$ satisfy the assumptions of Theorem \ref{convtojump} with $m = m^{(\alpha)}$ and 
	\begin{align}
		\lim_{\gamma \to \infty}(\tilde{b}_\gamma - b_\gamma) = 0, \label{eq42}
	\end{align}
	where $b_\gamma = -\int_{0}^{1}G_{m_{\gamma}}dj_\gamma$. It is easily checked that the assumption (i) of Theorem \ref{convtojump} holds. Then it is enough to show that the following hold:
 	\begin{align}
 		\lim_{\delta \to 0}\limsup_{\gamma \to \infty}\int_{0}^{\delta}m_\gamma(x)^2dx = 0, \label{eq41} \\
 		j_\gamma(dx) \xrightarrow[\gamma \to \infty]{w}0 \ \text{on} \ [1,\infty] \label{eq98}, \\
 		x j_\gamma(dx) \xrightarrow[\gamma \to \infty]{w} \kappa\delta_0(dx) \ \text{on} \ [0,1]. \label{eq43}
 	\end{align}
 	From Karamata's theorem \cite[Proposition 1.5.8]{Regularvariation}, we have 
 	\begin{align}
 		\lim_{\gamma \to \infty}\int_{0}^{\delta}m_\gamma(x)^2dx = (\alpha-1)^{-2}\int_{0}^{\delta}x^{2/\alpha -2}dx = \frac{1}{(2/\alpha -1)(\alpha-1)^{2}}\delta^{2/\alpha -1}. \label{}
 	\end{align}
 	Hence we obtain \eqref{eq41}. Next we prove \eqref{eq42}. By changing variables, we have
 	\begin{align}
 		\int_{1}^{\infty}|G_{m_\gamma}|j_\gamma(dx) = \frac{1}{\gamma^{1/\alpha -1}K(\gamma)}\int_{\gamma}^{\infty}|G_{m}|j(dx). \label{}
 	\end{align}
 	Again by Karamata's theorem, it holds that $\lim_{x \to \infty}\left|\frac{G_{m}(x)}{\alpha x m(x)}\right| = 1$. Then for any $\epsilon > 0$, there exists some $R > 0$ such that for every $x \geq R$ it holds that $\left|\frac{G_{m};x)}{\alpha x m(x)}\right| < 1 + \epsilon$. Then for $x \geq R$, it follows that 
 	\begin{align}
 		\frac{1}{\gamma^{1/\alpha -1}K(\gamma)}\int_{\gamma}^{\infty}|G_m|j(dx)
 		\leq \frac{-(1 + \epsilon)m(\gamma)}{\gamma^{1/\alpha -1}K(\gamma)}\int_{\gamma}^{\infty}xj(dx). \label{}
 	\end{align}
 	Then from assumption (i) and (ii), we have \eqref{eq42}. 
 	Next we show \eqref{eq98}.
 	By changing variables, for every $\delta > 0$, from the assumption (ii), it holds that
 	\begin{align}
 	\lim_{\gamma \to \infty}j_\gamma(\delta, \infty) \leq \lim_{\gamma \to \infty}\delta^{-1}\int_{\gamma \delta}^{\infty}xj(dx) = 0 \label{}
 	\end{align}
 	We prove \eqref{eq43}. By changing variables, for every bounded continuous function $f : [0,1] \to \bR$, it holds that
 	\begin{align}
 		\lim_{\gamma \to \infty}\int_{0}^{1}f(x)xj_\gamma(dx) &= \lim_{\gamma \to \infty}\int_{0}^{\gamma}f(\gamma^{-1}x)xj(dx) = \kappa f(0). \label{}
 	\end{align}
 	Hence \eqref{eq43} holds.
 	Next we show \eqref{eq110}.
 	From Theorem \ref{alpha1-2}, for every $\lambda > 0$, it holds that
 	\begin{align}
 	\gamma^{\alpha} \left(\chi_{m,j}\left(\frac{\lambda }{\gamma K(\gamma)}\right) - \frac{b\lambda }{\gamma K(\gamma)}\right) \xrightarrow{\gamma \to \infty} \kappa\lambda H_{m^{(\alpha)}}(\lambda). \label{}
 	\end{align}
 	Here $b = - \int_{0}^{\infty}G_m(x)j(dx)$.
 	Hence it holds that
 	\begin{align}
 	u(\lambda) := \chi_{m,j}(\lambda) - b\lambda \sim C\lambda^{\alpha}L^{\sharp}(1/\lambda)^{-\alpha} \ (\lambda \to +0) \label{eq111}
 	\end{align} 
 	where $C = \kappa H_{m^{(\alpha)}}(1)$. Since $u(\lambda) = \int_{0}^{\infty}\bP^m_x[1 - \mathrm{e}^{-\lambda T_0}]j(dx) - b\lambda$, the function $-u'(\lambda)$ is completely monotone.
	From the monotone density theorem\cite[Theorem 1.7.2b]{Regularvariation}), we obtain 
 	\begin{align}
 	u'(\lambda) &\sim C\alpha \lambda^{\alpha - 1}L^{\sharp}(1/\lambda)^{-\alpha} \ (\lambda \to +0). \label{eq112}
 	\end{align}
 	Let $\nu$ be a Radon measure on $(0,\infty)$ defined by
 	\begin{align}
 	\nu(dx) = n_{m,j}[T_0 > x] dx. \label{}
 	\end{align}
 	We note that $\nu(0,\infty) = P[T_0] = -\int_{0}^{\infty}G_m(x)j(dx) = b$ and 
 	\begin{align}
 	f(\lambda ) := \int_{0}^{\infty}\mathrm{e}^{-\lambda x} \nu(dx) = \frac{\chi_{m,j}(\lambda)}{\lambda}. \label{}
 	\end{align}
 	Then it holds that
 	\begin{align}
 	f'(\lambda) = \frac{\lambda\chi_{m,j}'(\lambda) - \chi_{m,j}(\lambda)}{\lambda^2} 
 	= \frac{\lambda u'(\lambda) - u(\lambda)}{\lambda^2}. \label{}
 	\end{align}
 	From \eqref{eq111} and \eqref{eq112}, we have
 	\begin{align}
 	f'(\lambda) \sim C (\alpha - 1) \lambda^{\alpha - 2}L^{\sharp}(1/\lambda)^{-\alpha} \ (\lambda \to 0) \label{}
 	\end{align} 
 	Then from Theorem \ref{KasaharaTauberian}, it follows that
 	\begin{align}
 	\nu [s, \infty) \sim \frac{\kappa \alpha^{\alpha - 1}}{(\alpha - 1)\Gamma (\alpha)}s^{-\alpha + 1}L^{\sharp}(s)^{-\alpha} \ (s \to \infty), \label{}
 	\end{align}
 	and from the monotone density theorem, we obtain \eqref{eq110}.	
 \end{proof}
 
 \section{Limit theorems for the occupation times of bilateral jumping-in diffusions}\label{section: arcsinelaw}
 
 Let $m_+,m_- \in \cM_1$ and $j_+,j_-$ be Radon measures on $(0,\infty)$ and suppose $(m_+,j_+)$ and $(m_-,j_-)$ satisfy (C). In this section, we treat bilateral jumping-in diffusion processes i.e. Markov processes on $\bR$ which behaves like $X_{m_+,j_+}$ while $X$ is positive and like $-X_{m_-,j_-}$ while $X$ is negative and as soon as the process hits the origin it is thrown into $\bR \setminus \{ 0 \}$ according to $j_+$ and $j_-$.
 The precise definition is as follows. 
 Take two independent Poisson random measures $N_+$ and $\tilde{N}_-$ whose intensity measures are $n_{m_+,j_+}$ and $n_{m_-,j_-}$ on a common probability space. 
 Define $N_{m_-,j_-}(ds de) = \tilde{N}_{m_-,j_-}(ds d(-e))$. Then we define bilateral jumping-in diffusion process $X_{m_+,j_+;m_-,j_-}$ from the excursion point process $N_{m_+,j_+} + N_{m_-,j_-}$.
 
 Define 
 \begin{align}
 	A(t) = \int_{0}^{t}1_{[0,\infty)}(X_{m_+,j_+;m_-,j_-}(s))ds \label{}
 \end{align}
 for $t \geq 0$ and we study the fluctuation of the mean occupation time on the positive side $(1/t)A(t)$ as $t \to \infty$, in the case the limit degenerates, that is,
 \begin{align}
 	\lim_{t \to \infty}\frac{1}{t}A(t) \xrightarrow[t \to \infty]{P} p \in (0,1). \label{}
 \end{align}

\begin{Thm}\label{flucofA alpha1-2}
	Let $m_+,m_- \in \cM_1$ and $j_+,j_-$ be Radon measures on $(0,\infty)$. Suppose $(m_+,j_+)$ and $(m_-,j_-)$ satisfy (C).
	Assume the following hold:
	\begin{enumerate}
		\item $m_\pm(x,\infty)  \sim w_{\pm}(\alpha - 1)^{-1}x^{1/\alpha -1}K (x) \  (x \to \infty)$ for  constants $\alpha  \in (1,2)$, $w_{\pm} > 0$ and a slowly varying function $K$ at $\infty$, respectively,
		\item $\kappa_\pm:= \int_{0}^{\infty}xj_\pm(dx) < \infty$.
	\end{enumerate}
	Then we have
	\begin{align}
	g(\gamma)\left(\frac{A(\gamma t)}{\gamma} - p t\right) \xrightarrow[\gamma \to \infty ]{f.d.} (1-p)w_+T(m^{(\alpha)};\tilde{\kappa}_+t) - pw_-\tilde{T}(m^{(\alpha)};\tilde{\kappa}_-t), \label{eq157}
	\end{align}
	where 
	\begin{align}
	b_\pm &= 
	\int_{0}^{\infty}j_\pm(dx)\int_{0}^{x}m_\pm(y,\infty)dy,\quad p = \frac{b_+}{b_+ + b_-}, \quad g(\gamma) = \frac{\gamma}{\gamma^{1/\alpha}K(\gamma)}, \ \tilde{\kappa}_\pm = \frac{\kappa_\pm}{b_+ + b_-} \label{}
	\end{align}
	and, $T(m^{(\alpha)};t)$ and $\tilde{T}(m^{(\alpha)};t)$ are i.i.d.
\end{Thm}
\begin{proof}
	Step1: 
	Define the local time $\ell(t)$ at $0$ of $X_{m_+,j_+;m_-,j_-}$ as the right-continuous inverse of the process $\eta_{m_+,j_+} + \eta_{m_-,j_-}$.
	It is not difficult to check that for every $t > 0$, it holds that
	\begin{align}
	&\eta_{m_+,j_+}(\ell(t)-) \leq A(t) \leq \eta_{m_+,j_+}(\ell(t)) \label{} \\
	&\eta_{m_-,j_-}(\ell(t)-) \leq t - A(t) \leq \eta_{m_-,j_-}(\ell(t)) \label{}
	\end{align}
	Therefore the process $A(t) - pt = (1-p)A(t) - p(t - A(t))$ satisfies the following:
	\begin{align}
	&(1-p)\eta_{m_+,j_+}(\ell(t)-) - p\eta_{m_-,j_-}(\ell(t)) \label{} \\
	\leq& A(t) - pt \label{} \\
	\leq& (1-p)\eta_{m_+,j_+}(\ell(t)) - p\eta_{m_-,j_-}(\ell(t)-). \label{}
	\end{align}
	Hence for every $t \geq 0$, we have 
	\begin{align}
	&| (A( t) - p t) - ((1-p)\eta_{m_+,j_+}(\ell( t)) - p \eta_{m_-,j_-}(\ell( t) ))| \leq \Delta\eta(\ell(t)), \label{eq158}
	\end{align}
	where $\eta = \eta_{m_+,j_+} + \eta_{m_-,j_-}$ and $\Delta\eta(t) = \eta(t) - \eta(t-)$.
	
	Step2: Let us show the following:
	\begin{align}
	\frac{1}{\gamma^{1/\alpha}K(\gamma)}\Delta\eta(\ell(\gamma t)) 
	\xrightarrow[\gamma \to \infty]{f.d.} 0. \label{}
	\end{align}
	For this, it is enough to show
	\begin{align}
	\frac{1}{\gamma^{1/\alpha}K(\gamma)}\Delta\eta(\ell(\gamma t)) \xrightarrow[\gamma \to \infty]{P} 0 \ \text{for every} \ t \geq 0. \label{eq114}
	\end{align}
	Take $\epsilon > 0$ and $\delta > 0$ and define $c_{\pm}(\delta) = \frac{1 \pm \delta}{b_+ + b_-}$. Then it holds that
	\begin{align}
	&P\left[\Delta\eta(\ell(\gamma t)) > \gamma^{1/\alpha}K(\gamma)\epsilon\right] \label{} \\
	\leq& P[\ell(\gamma t) \not\in [c_-(\delta)\gamma t, c_+(\delta)\gamma t]] \label{eq108} \\
	&+ P\left[\sup_{s \in [c_-(\delta)\gamma t, c_+(\delta)\gamma t]}\Delta\eta(s) > \gamma^{1/\alpha}K(\gamma)\epsilon \right]. \label{eq107}
	\end{align}
	By the definition of $\ell(t)$, we have
	\begin{align}
	P[\ell(\gamma t) \not\in [c_-(\delta)\gamma t, c_+(\delta)\gamma t]]
	&\leq P[\eta(c_-(\delta)\gamma t) > \gamma t] + P[\eta(c_+(\delta)\gamma t) \leq \gamma t ]. \label{eq113} 
	\end{align}
	From Theorem \ref{alpha1-2}, we have $\eta(\gamma)/\gamma \xrightarrow[\gamma \to \infty]{P} b_+ + b_-$.
	Therefore we have $\eqref{eq113} \to 0 \ (\gamma \to \infty)$.
	Recall that $N := N_{m_+,j_+} + N_{m_-,j_-}$ is the excursion point process of $X_{m_+,j_+;m_-,j_-}$. Define $n = n_{m_+,j_+} + n_{m_-,j_-}$. Then we have
	\begin{align}
	&P\left[\sup_{s \in [c_-(\delta)\gamma t, c_+(\delta)\gamma t]}\Delta\eta(s) > \gamma^{1/\alpha}K(\gamma)\epsilon \right] \label{} \\
	=& 1 - P\left[ N\{ (s,e) \in [0,\infty) \times \bE \mid s \in [c_-(\delta)\gamma t, c_+(\delta)\gamma t],\ T_0(e) > \gamma^{1/\alpha}K(\gamma)  \} =0 \right] \label{} \\
	=& 1 - \exp (-\gamma t(c_+(\delta) - c_-(\delta)))n [T_0 > \gamma^{1/\alpha}K(\gamma)\epsilon]). \label{}
	\end{align}
	Therefore from Theorem \ref{alpha1-2}, we have
	\begin{align}
	\limsup_{\gamma \to \infty}P\left[\Delta\eta(\ell(\gamma t)) > \gamma^{1/\alpha}K(\gamma)\epsilon\right] \leq 1 - \exp \left( - t (c_+(\delta) - c_-(\delta)) \frac{\kappa \alpha^{\alpha - 1}}{\Gamma (\alpha)}\epsilon^{-\alpha} \right). \label{} 
	\end{align}
	Since $c_+(\delta) - c_-(\delta) \to 0$ as $\delta \to +0$, we obtain \eqref{eq114}.
	
	Step3: 
	We now reduce \eqref{eq157} to 
	\begin{align}
	g(\gamma)\tilde{\eta}(\ell (\gamma t)) &= 
	g(\gamma)\left((1-p)\frac{\eta_{m_+,j_+}(\ell(\gamma t))}{\gamma} - p \frac{\eta_{m_-,j_-}(\ell(\gamma t))}{\gamma}\right) \label{}  \\
	&\xrightarrow[\gamma \to \infty]{f.d.} (1-p)c_+T(m^{(\alpha)};\kappa_+t) - pc_-\tilde{T}(m^{(\alpha)};\kappa_-t), \label{}
	\end{align}
	and this argument is due to \cite{KasaharaKotani:Onlimit} and \cite{KasaharaWatanabe:Brownianrepresentation}.
	From Theorem \ref{alpha1-2}, we have 
	\begin{align}
	&\left(g(\gamma)\tilde{\eta}(\gamma t), \frac{1}{\gamma}\eta(\gamma t)\right) \label{} \\ 
	\xrightarrow[\gamma \to \infty]{d}& ((1-p)w_+T(m^{(\alpha)};\kappa_+t) - pw_-\tilde{T}(m^{(\alpha)};\kappa_-t), (b_+ + b_-)t) \ \text{on} \ \bD \ \text{in} \ J_1\text{-topology}. \label{}
	\end{align}
	Since the right-continuous inverse process of $\frac{1}{\gamma}\eta(\gamma t)$ is $\frac{1}{\gamma}\ell(\gamma t)$, it follows that
	\begin{align}
	g(\gamma)\tilde{\eta}(\ell (\gamma t)) 
	\xrightarrow[\gamma \to \infty]{d} (1-p)w_+T\left(m^{(\alpha)};\tilde{\kappa}_+t\right) - pw_-\tilde{T}\left(m^{(\alpha)};\tilde{\kappa}_-t\right) \ \text{on} \ \bD \ \text{in} \ M_1\text{-topology}. \label{}
	\end{align}
	See e.g.\cite{Whitt:Stochastic-process} on $M_1$-convergence.
	Since the limit process has no fixed jumps, $M_1$-convergence implies the finite-dimensional convergence.
\end{proof}

\section{The case: $\alpha = 2$}\label{section: alpha2}

In this section, we prove the result shown in Section \ref{section: scalinglimit} and \ref{section: arcsinelaw} in the case when the scale parameter $\alpha = 2$.

\begin{Thm}\label{alpha2}
	Let $m \in \cM_1$ and $j$ be a Radon measure on $(0,\infty)$ and assume $(m,j)$ satisfies $\mathrm{(C)}$.
	Suppose the following hold: 
	\begin{enumerate}
		\item The function $K(\gamma) = \left(\int_{0}^{\gamma}m(x,\infty)^2dx \right)^{1/2} $ varies slowly at $\infty$,
		\item $\kappa := \int_{0}^{\infty}xj(dx) < \infty$,
		\item $\kappa^0 := \int_{0}^{\infty}j(dx) \int_{0}^{x}dy\int_{0}^{y}dm(z)\int_{0}^{z}m(w,\infty) dw< \infty$,
		\item $\int_{1}^{\infty}G^1_m(x)j(dx) < \infty$.
	\end{enumerate}
	Then we have
	\begin{align}
	\frac{1}{\gamma^{1/2} K(\gamma)}(\eta_{m,j}(\gamma t) + b\gamma t) \xrightarrow[\gamma \to \infty]{d}B(\tilde{\kappa} t) \ \text{on} \ \bD, \label{eq145}
	\end{align}
	and 
	\begin{align}
	n_{m,j}[T_0 > s ] = o(s^{-2}L^{\sharp}(s)^{-2}) \ (s \to \infty), \label{eq143}
	\end{align}
	where $L^{\sharp}$ be a de Bruijn conjugate of $L(s) = K(s^2)$, $b = \int_{0}^{\infty}j(dx)\int_{0}^{x}m(y,\infty)dy$ and $\tilde{\kappa} = 2\kappa - \frac{2\kappa^0}{K(\infty)^2}$.
\end{Thm}
\begin{proof}  
	We may assume $m(\infty) = 0$ without loss of generality.
	Define
	\begin{align}
	m_\gamma(x) = \frac{\gamma^{1/2}}{K(\gamma)}m(\gamma x) \  \text{and} \ j_\gamma(dx) = \gamma j(d(\gamma x)). \label{}
	\end{align}
	Then we have
	\begin{align}
	\frac{1}{\gamma^{1/2} K(\gamma)}(\eta_{m,j}(\gamma t) - b\gamma t)
	\overset{d}{=} \eta_{m_\gamma,j_\gamma}(t) - b_\gamma t \ \text{on} \ \bD. \label{}
	\end{align}
	Then we show that $\{ m_\gamma \}_\gamma$ and $\{ j_\gamma \}_\gamma$ satisfy the assumptions of Theorem \ref{convtoBM} in the case:
	\begin{align}
	m \equiv 0,\ \kappa_1 = \int_{0}^{\infty}xj(dx),\ \kappa_2 = -\frac{1}{K(\infty)^2}\int_{0}^{\infty}j(dx) \int_{0}^{x}dy\int_{0}^{y}G_m(z)dm(z). \label{}
	\end{align}
	and the following holds:
	\begin{align}
	\lim_{\gamma \to \infty}\left( \tilde{b}_\gamma + \int_{0}^{1}G_{m_\gamma}(x)j_\gamma(dx) \right) = 0. \label{eq62}
	\end{align}
	It is easily checked that the assumptions (iii) and (iv) of Theorem \ref{convtoBM} hold.
	Since it holds that for every $0 \leq a < b$
	\begin{align}
	\int_{a}^{b}m_\gamma(x)^2dx = \frac{1}{K(\gamma)^2}\int_{\gamma a}^{\gamma b}m(y)^2dy, \label{}
	\end{align}
	then from the assumption (i), we have 
	\begin{align}
	\lim_{\gamma \to \infty}m_\gamma(x) = 0 \  \text{for every}\  x > 0, \label{} \\
	\lim_{\gamma \to \infty}\int_{0}^{\delta}m_\gamma(x)^2dx = 1 \ \text{for every} \   \delta > 0. \label{}
	\end{align}
	Then it is enough to show that the following and \eqref{eq62} hold:
	\begin{align}
	-s\bullet m_\gamma \bullet G_{m_\gamma}j_\gamma(dx) &\xrightarrow[\gamma \to \infty]{w}\kappa_2 \delta_0(dx) \ \text{on} \ [0,1]. \label{eq63}
	\end{align}
	Here $\kappa_2 = -\frac{1}{K(\infty)}\int_{0}^{\infty}j(dx)\int_{0}^{x}dy\int_{0}^{y}G_mdm(z)$.
	First, we prove \eqref{eq62}. From the change of variables and Schwarz's inequality, we have
	\begin{align}
	|G_{m_\gamma}| = -\frac{1}{\gamma^{1/2} K(\gamma)}\int_{0}^{\gamma x}m(y)dy
	\leq x^{1/2} \frac{K(\gamma x)}{K(\gamma)}. \label{}
	\end{align}
	Since $K$ varies slowly at $\infty$, from Potter's theorem \cite[Theorem 1.5.6]{Regularvariation}, we have $\frac{K(\gamma x)}{K(\gamma )} \leq 2 x^{1/2}$ for large $\gamma > 0$ and every $x \geq 1$. Hence it follows that
	\begin{align}
	|G_{m_\gamma}| \leq 2x \ \text{for large}\  \gamma  > 0 \ \text{and every}\  x \geq 1. \label{}
	\end{align}
	Hence we obtain \eqref{eq62}. Next we prove \eqref{eq63}.
	Let $f: [0, 1] \to \bR$ be a bounded, continuous function. By the change of variables, we have
	\begin{align}
	&\int_{0}^{1}f(x)(-s\bullet m_\gamma \bullet G_{m_\gamma})j_\gamma(dx) \label{} \\
	=& \gamma\int_{0}^{\gamma}f(\gamma^{-1}x)(-s\bullet m_\gamma \bullet G_{m_\gamma}(\gamma^{-1}x))j(dx)  \label{} \\
	=& \frac{1}{K(\gamma)^2}\int_{0}^{\gamma}f(\gamma^{-1}x)\left( -\int_{0}^{x}dy\int_{0}^{y}G_mdm(z) \right) j(dx). \label{eq74}
	\end{align}
	From the dominated convergence theorem, we obtain \eqref{eq63}.
	Finally, we show \eqref{eq143}.
	Define 
	\begin{align}
	\nu(ds) =  \left(\int_{s}^{\infty}n_{m,j}[T_0 > u]du \right)ds \quad \text{and} \quad \hat{\nu}(\lambda) = \int_{0}^{\infty}\mathrm{e}^{-\lambda s}\nu(ds) = - \frac{u(\lambda)}{\lambda^2}, \label{}
	\end{align}
	where $u(\lambda) = \chi_{m,j}(\lambda) -b\lambda$.
	From \eqref{eq145}, it holds that $u(\lambda) \sim -\kappa\lambda^2L^{\sharp}(1/\lambda)^{-2} \ (\lambda \to +0 )$.
	Hence it follows that
	\begin{align}
	(f(1/(\gamma t)) - f(1/\gamma))L^{\sharp}(\gamma)^2 \xrightarrow{\gamma \to \infty} 0. \label{}
	\end{align}
	Then from Tauberian theorem \cite[Therem 3.9.1]{Regularvariation}, it holds for every $s > 1$,
	\begin{align}
	L^{\sharp}(\gamma)^2\nu (\gamma, \gamma s] \xrightarrow{\gamma \to \infty} 0. \label{}
	\end{align}
	Then from the monotone density theorem \cite[Theorem 1.7.2, 3.6.8]{Regularvariation}, we obtain \eqref{eq143}.
\end{proof}
\begin{Rem}
	Since it holds that $L^{\sharp}(\gamma)L(\gamma L^{\sharp}(\gamma)) \to 1 \ (\gamma \to \infty)$, when $K(\gamma) \to \infty \ (\gamma \to \infty)$, we have $L^{\sharp}(\gamma) \to 0 \ (\gamma \to \infty)$ and when $K(\gamma) \to c \in (0,\infty) \ (\gamma \to \infty)$, we have $L^{\sharp}(\gamma) \to 1/c \ (\gamma \to \infty)$. 
\end{Rem}
Applying Theorem \ref{alpha2}, the following theorem can be proved by the same argument as Theorem \ref{flucofA alpha1-2}, so we omit the proof.
\begin{Thm}\label{flucofA alpha2}
	Let $m_+,m_- \in \cM_1$ and $j_+,j_-$ be Radon measures on $(0,\infty)$. Assume $(m_+,j_+)$ and $(m_-,j_-)$ satisfy (C).
	Suppose the following conditions hold: 
	\begin{enumerate}
		\item $K_\pm(x) := \left( \int_{0}^{x}m_\pm(y,\infty)^2dy\right)^{1/2} \sim c_\pm K(x) \ (x \to \infty) $ for constants $c_\pm > 0$ and a slowly varying function $K$ at $\infty$,
		\item $\kappa_\pm := \int_{0}^{\infty}xj_\pm(dx) < \infty$
		\item $\kappa^0_\pm := \int_{0}^{\infty}j_\pm(dx) \int_{0}^{x}dy\int_{0}^{y}dm_\pm(z)\int_{0}^{z}m_\pm(w,\infty)dw < \infty$,
		\item $\int_{1}^{\infty}G^1_{m_\pm}(x)j_\pm(dx) < \infty$.
	\end{enumerate} 
	Then we have
	\begin{align}
	g(\gamma)\left(\frac{A(\gamma t)}{\gamma} - p t\right) \xrightarrow[\gamma \to \infty]{f.d.} (1-p)B(\tilde{\kappa}_+t) - p\tilde{B}(\tilde{\kappa}_-t), \label{}
	\end{align}
	where  
	\begin{align}
	a_\pm &= \int_{0}^{\infty}j_\pm(dx)\int_{0}^{x}m_\pm(y,\infty)dy,\quad p = \frac{a_+}{a_+ + a_-},\  g(\gamma) = \frac{\gamma}{\gamma^{1/2}K(\gamma)}, \label{} \\
	\tilde{\kappa}_\pm &= 2\kappa - \frac{2\kappa^0_\pm}{K_\pm(\infty)}, \label{}
	\end{align}
	and, $B$ and $\tilde{B}$ are independent Brownian motions.
\end{Thm}

\section{The case: $\alpha = 1$}\label{section: alpha1}
In this section, we prove the result shown in Section \ref{section: scalinglimit} and \ref{section: arcsinelaw} in the case when the scale parameter $\alpha = 1$.

To obtain the tail behavior of $\eta_{m,j}$ in the case $\alpha = 1$, we need the following version of Tauberian theorem.
 \begin{Thm}{\label{KasaharaTauberian2}}
	Let $\mu$ be a Radon measure on $[0,\infty)$ and $K$ be a slowly varying function at $\infty$. Define 
	\begin{align}
	\hat{\mu}(\lambda) = \int_{0-}^{\infty}\mathrm{e}^{-\lambda x}\mu(dx) \ (\lambda > 0). \label{}
	\end{align}
	Then the following are equivalent:
	\begin{enumerate}
		\item $\frac{\mu [x,\theta x)}{K(x)} \sim \log \theta \ (x \to \infty)$ for every $\theta > 1$,
		\item $-\hat{\mu}'(\lambda) \sim \frac{K(1/\lambda)}{\lambda} \ (\lambda \to +0)$.
	\end{enumerate}
\end{Thm}
The proof of Theorem \ref{KasaharaTauberian2} can be found in \cite[Theorem 2.2]{Kasahara:Tailsof}.
  \begin{Thm}\label{alpha1}
	Let $m \in \cM_1$, $j$ be a Radon measure on $(0,\infty)$ and $K$ be a slowly varying function at $\infty$ such that $K$ and $1/K$ are locally bounded on $[0,\infty)$ and assume $(m,j)$ satisfies $\mathrm{(C)}$.
	Suppose the following conditions hold: 
	\begin{enumerate}
		\item $\lim_{\gamma \to \infty}\frac{m(\gamma x) - m(\gamma)}{K(\gamma)} = \log x$ for every $x > 0$,
		\item $\kappa := \int_{0}^{\infty}xj(dx) < \infty$.
	\end{enumerate} 
	Then if we take $\kappa = \int_{0}^{\infty}xj(dx)$ and $b_\gamma = -\int_{0}^{\gamma}j(dx)\int_{0}^{x}(m(y) - m(\gamma))dy$, we have
	\begin{align}
	\frac{1}{\gamma K(\gamma)}( \eta_{m,j}(\gamma t) - b_\gamma \gamma t )  \xrightarrow[\gamma \to \infty]{d} T(m^{(1)};\kappa t) \ \text{on} \ \bD. \label{eq146}
	\end{align}
	In addition to (i) and (ii), assume
	\begin{enumerate}
		\setcounter{enumi}{2}
		\item $j(x,\infty) \leq C x^{-1-\beta} \ (x \geq 1)$ for constants $C > 0$ and $\beta > 0$.
		\end{enumerate}
		Then we can replace $b_\gamma$ in \eqref{eq146} to $b^{\infty}_\gamma := - \int_{0}^{\infty}j(dx)\int_{0}^{x}(m(y) - m(\gamma))dy$ and the following holds:
		\begin{align}
		n_{m,j}[T_0 > s] &\sim \kappa s^{-1} K^{\sharp}(s)^{-1} \ (s \to \infty) \label{eq147}
		\end{align}
		where $K^{\sharp}$ is a de Bruijn conjugate of $K$.
	In addition to (i) and (ii), assume
	\begin{enumerate}
		\setcounter{enumi}{3}
		\item $\lim_{x \to \infty}m(x) = \infty$
		\item $\int_{1}^{\infty}G^1_m(x)j(dx) < \infty$
	\end{enumerate}
	Then we have 
	\begin{align}
	\frac{1}{\gamma m(\gamma)}\eta_{m,j} (\gamma) \xrightarrow[\gamma \to \infty]{P} \kappa. \label{eq134}
	\end{align}
\end{Thm}
\begin{proof}
	Define
	\begin{align}
	m_\gamma(x) = \frac{m(\gamma x) - m(\gamma)}{K(\gamma)},\  j_\gamma(dx) = \gamma j(d(\gamma x)), \  \tilde{b}_\gamma = -\int_{0}^{1}G_{m_\gamma}j_\gamma (dx). \label{}
	\end{align}
	Then  we have
	\begin{align}
	\frac{1}{\gamma K(\gamma)}(\eta_{m,j}(\gamma t) - b_\gamma \gamma t)
	\overset{d}{=} \eta_{m_\gamma,j_\gamma}(t) - \tilde{b}_\gamma t \ \text{on} \ \bD \label{}
	\end{align}
	Then we show that $\{ m_\gamma \}_\gamma$ and $\{ j_\gamma \}_\gamma$ satisfy the assumptions of Theorem \ref{convtojump} with $m = m^{(1)}$.
	It is easily checked that the assumptions (i) and (iii) of Theorem \ref{convtojump} hold. Moreover, the assumption (iv) of Theorem \ref{convtojump} can be checked by the same way as Theorem \ref{alpha1-2}. Then in order to show \eqref{eq146}, it is enough to show that the following holds:
	\begin{align}
	\lim_{a \to +0}\limsup_{\gamma \to \infty}\int_{0}^{a}m_\gamma(x)^2dx = 0. \label{eq80}
	\end{align}
	From Potter's theorem\cite[Theorem 3.8.6]{Regularvariation}, for every $\epsilon > 0$, there exists a constant $C_\epsilon > 0$ such that the following hold:
	\begin{align}
	\frac{m(\gamma x) - m(\gamma)}{K(\gamma )} \leq C_\epsilon \max \{ x^{\epsilon}, x^{-\epsilon} \} \ \text{for every}\ x > 0 \ \text{and} \ \gamma > 0. \label{eq83}
	\end{align} 
	From \eqref{eq83}, we easily obtain \eqref{eq80} and therefore \eqref{eq146} holds.
	We assume the assumption (iii) holds.
	Then from \eqref{eq83}, it is not difficult to see that for constants $C' > 0$ and $\beta' \in (0,\beta)$, 
	\begin{align}
	\frac{1}{K(\gamma)}\int_{\gamma}^{\infty}j(dx)\int_{0}^{x}(m(y) - m(\gamma))dy
	\leq C' \int_{\gamma}^{\infty}x^{1 + \beta'}j(dx) \xrightarrow{\gamma \to \infty} 0, \label{}
	\end{align}
	and therefore $(b_\gamma - b^\infty_\gamma)/ K(\gamma) \xrightarrow{\gamma \to \infty} 0$. Next we show \eqref{eq147}
	Define
	\begin{align}
	\nu(ds) = n_{m,j}[T_0 > s]ds \quad \text{and} \quad \hat{\nu}(\lambda) = \int_{0}^{\infty}\mathrm{e}^{-\lambda s}\nu(ds) = \frac{\chi_{m,j}(\lambda)}{\lambda}. \label{eq141}
	\end{align}	
	If we can show 
	\begin{align}
	-\hat{\nu}'(\lambda) \sim \frac{\kappa}{\lambda K^{\sharp}(1/\lambda)} \ (\lambda \to +0 ), \label{eq137}
	\end{align}
	we have from \eqref{eq141} and Theorem \ref{KasaharaTauberian2}, for every $s > 0$,
	\begin{align}
	K^{\sharp}(\gamma)\int_{\gamma}^{\gamma s}n_{m,j}[T_0 > u]du \sim \kappa \log s \ (\gamma \to \infty). \label{}
	\end{align}
	Then from the monotone density theorem, the desired result is shown. Hence it is enough to show \eqref{eq137}. 
	From \eqref{eq146}, we have
	\begin{align}
	\gamma \left(\chi_{m,j}\left(\frac{1 }{\gamma K(\gamma)}\right) - \frac{b^\infty_\gamma  }{\gamma K(\gamma)}\right) \xrightarrow{\gamma \to \infty} C. \label{eq138}
	\end{align}
	where $C = \kappa H_{m^{(1)}}(1)$.
	Define $\varTheta(\gamma) = \gamma\chi_{m,j}(1/\gamma)$. Then from \eqref{eq138} and the assumption (i), it holds that for every $s > 0$,
	\begin{align}
	\frac{\varTheta(\gamma sK(\gamma s)) - \varTheta (\gamma K(\gamma))}{K(\gamma)} \xrightarrow{\gamma \to \infty} \kappa \log s. \label{}
	\end{align}
	Hence we obtain for every $s > 0$, 
	\begin{align}
	K^{\sharp}(\gamma)(\varTheta(\gamma s) - \varTheta(\gamma)) \xrightarrow{\gamma \to \infty} \kappa \log s, \label{}
	\end{align}
	Since we have
	\begin{align}
	\varTheta''(\gamma) &= \frac{\chi_{m,j}''(1/\gamma)}{\gamma^3}  = -\frac{1}{\gamma^3}\int_{0}^{\infty}\bP^m_x[T^2_0\mathrm{e}^{-T_0/\gamma}]j(dx) \leq 0, \label{}
	\end{align}
	it holds that from the monotone density theorem,
	\begin{align}
	\varTheta'(\gamma) \sim \frac{\kappa}{\gamma K^{\sharp}(\gamma)} \ ( \gamma \to \infty). \label{}
	\end{align}
	Then we obtain
	\begin{align}
	-\hat{\nu}'(\lambda) = \frac{\varTheta'(1/\lambda)}{\lambda^2} 
	\sim  \frac{\kappa}{\lambda K^{\sharp}(1/\lambda)} \ (\gamma \to \infty), \label{}
	\end{align}
	Next we show \eqref{eq134}.
	Since $\lim_{x \to \infty}m(x) = \infty$, it holds that
	\begin{align}
	\lim_{\gamma \to \infty}\frac{b_\gamma}{m(\gamma)} = \kappa. \label{eq133}
	\end{align} 
	Then from Theorem 3.7.4 in \cite{Regularvariation}, we have $\lim_{\gamma \to \infty}m(\gamma) / K(\gamma) = \infty$ and therefore from \eqref{eq146},
	we obtain \eqref{eq134}	and the proof is complete. 
\end{proof}

\begin{Thm}\label{flucofA alpha1}
	Let $m_+,m_- \in \cM_1$, $j_+,j_-$ be Radon measures on $(0,\infty)$ and $K$ be a slowly varying function at $\infty$ such that $K$ and $1/K$ are locally bounded on $[0,\infty)$. Assume $(m_+,j_+)$ and $(m_-,j_-)$ satisfy (C).
	Suppose the following conditions hold: 
	\begin{enumerate}
		\item $\lim_{x \to \infty}\frac{m_\pm(\gamma x) - m_\pm(\gamma)}{K(\gamma)} = w_\pm\log x \ ( x > 0)$ for constants $w_\pm > 0$,
		\item $j_\pm(x,\infty) \leq Cx^{-1 - \beta} \ (x \geq 1)$ for constants $C>0$ and $\beta > 0$.
	\end{enumerate} 
	Then we have $\lim_{\gamma \to \infty}p(\gamma) = p$ and
	\begin{align}
		g(\gamma)\left(\frac{A(q(\gamma) t)}{\gamma} - \frac{p(\gamma)q(\gamma)}{\gamma
		} t\right) \xrightarrow[\gamma \to \infty ]{f.d.} (1-p)w_+T(m^{(1)};\tilde{\kappa}_+t) - pw_-\tilde{T}(m^{(1)};\tilde{\kappa}_-t).  	 \label{}
\end{align}	where $T(m^{(1)};t)$ and $\tilde{T}(m^{(1)};t)$ are i.i.d. and
	\begin{align}
	&b_\pm(\gamma) = -\int_{0}^{\infty}j_\pm(dx)\int_{0}^{x}({m_\pm}(y) - m_\pm(\gamma))dy,\  p(\gamma) = \frac{b_+(\gamma)}{b_+(\gamma) + b_-(\gamma)},\label{} \\  
	&\kappa_\pm = \int_{0}^{\infty}xj_\pm(dx), \ g(\gamma) = \frac{1}{ K(\gamma)}, \label{}
	\end{align}
	and when $\lim_{x \to \infty}m_\pm(x) < \infty$,
	\begin{align}
		b_\pm = \int_{0}^{\infty}j_\pm(dx)\int_{0}^{x}(m_\pm(\infty) - m_\pm(y))dy, \ 	\tilde{\kappa}_\pm = \frac{\kappa_\pm}{b_+ + b_-}, \  p = \frac{b_+}{b_+ + b_-},\ q(\gamma) = \gamma, \label{}
	\end{align}
	and when $\lim_{x \to \infty}m_\pm(x) = \infty$, 
	\begin{align}
		&a_\pm = \frac{\kappa_\pm w_\pm}{w_+ + w_-}, \  \tilde{\kappa}_\pm = \frac{\kappa_\pm}{a_+ + a_-}, \ 
		p = \frac{a_+}{a_+ + a_-}, \ q(\gamma) = \gamma (m_+(\gamma) + m_-(\gamma)). \label{}
\end{align}
\end{Thm}

\begin{proof}
	In the positive recurrent case, that is, when $\lim_{x \to \infty}m_\pm(x) = \infty$ holds, we can prove almost the same argument as Theorem \ref{flucofA alpha1-2}. Hence we only show the case $\lim_{x \to \infty}m_\pm(x) = \infty$.
	
	From \eqref{eq158} we have
	\begin{align}
	&|(A_+(q(\gamma) t) - p(\gamma)q(\gamma) t) - ((1-p(\gamma))\eta_{m_+,j_+}(\ell(q(\gamma) t)) - p(\gamma) \eta_{m_-,j_-}(\ell(q(\gamma) t) ))| \label{} \\
	\leq &\Delta\eta(\ell(q(\gamma) t)). \label{}
	\end{align}
	Then from the same argument in Theorem \ref{flucofA alpha1-2}, it is enough to show the following:
	\begin{align}
		&\lim_{\delta' \to +0}\limsup_{\gamma \to \infty}\bP\left[\sup_{s \in [c_-(\delta')\gamma t, c_+(\delta')\gamma t]}\Delta\eta(s) > \gamma K(\gamma)\epsilon\right] = 0 \ \text{for every} \ \epsilon > 0, \label{eq149} \\
		&\left(g(\gamma)\tilde{\eta}(\gamma t), \frac{1}{q(\gamma)}\eta(\gamma t)\right) \nonumber \\ 
		\xrightarrow[\gamma \to \infty]{d}& ((1-p)w_+T(m^{(1)};\kappa_+t) - pw_-\tilde{T}(m^{(1)};\kappa_-t), (a_+ + a_-)t) \ \text{on} \ \bD \ \text{in} \ J_1\text{-topology}. \label{eq151} \\
		&\lim_{\gamma \to \infty}p(\gamma) = p, \label{eq152} \\
		&\lim_{\gamma \to \infty}\bP[\ell(q(\gamma)t) \not\in [c_-(\delta)\gamma t, c_+(\delta)\gamma t]] = 0 \ \text{for every} \ \delta > 0, \label{eq148}
	\end{align}
	where $\tilde{\eta}(t) = (1 - p)\eta_{m_+,j_+}(t) - p \eta_{m_-,j_-}(t)$ and $c_\pm(\delta) = \frac{1 \pm \delta}{a_+ + a_-}$.
	From Theorem \ref{alpha1}, we easily obtain \eqref{eq149} and \eqref{eq151}.
	Since we have the assumptions (i) and (ii), it follows from de Haan's theorem \cite[Theorem 3.7.3]{Regularvariation},
	\begin{align}
		\lim_{\gamma \to \infty}\frac{m_+(\gamma)}{m_-(\gamma)} = \frac{w_+}{w_-}. \label{}
	\end{align}
	Moreover, from \eqref{eq134}, it holds that $\lim_{\gamma \to \infty}b_\pm(\gamma)/m_\pm(\gamma) = \kappa_\pm$. Hence we obtain \eqref{eq152}.
	We prove \eqref{eq148}.
	Note that
	\begin{align}
	P[\ell(q(\gamma) t) \not\in [c_-(\delta)\gamma t, c_+(\delta)\gamma t]]
	&\leq P[\eta(c_-(\delta)\gamma t) > q(\gamma) t] + P[\eta(c_+(\delta)\gamma t) \leq q(\gamma) t ]. \label{eq150} 
	\end{align}
	Then from Theorem \ref{alpha1}, we have $\lim_{\gamma \to \infty}\eta(\gamma)/q(\gamma) = \frac{w_+\kappa_+ + w_-\kappa_-}{w_+ + w_-}$
	and therefore $\eqref{eq150} \to 0 \ (\gamma \to \infty)$.
\end{proof}

\appendix

\section{Appendix: Continuity theorem for Laplace transforms of spectrally positive L\'evy processes}\label{appendix: contthmofLT}
	\begin{Prop}\label{continuityofLT}
		Let $X_n,X$ be real-valued random variables with infinitely divisible laws such that their L\'evy measures are supported on $(0,\infty)$.
		Then $X_n \xrightarrow[n \to \infty]{d} X$ holds if and only if $\lim_{n \to \infty} E[\mathrm{e}^{-\lambda X_n}] = E[\mathrm{e}^{-\lambda X}]$ holds for every $\lambda > 0$.
	\end{Prop}
	\begin{proof}
		The proof of the direct assertion can be found in \cite[Appendix]{KasaharaWatanabe:Remarkson}. Hence we prove only the inverse assertion. Since $X_n,X$ have infinite divisible laws and their L\'evy measures are supported on $(0,\infty)$, their characteristic exponent $\psi_n,\psi$ are represented as
		\begin{align}
			\psi_n(\xi) &= \log E[\mathrm{e}^{i \xi X_n}] = i c_n \xi -\frac{1}{2}a_n\xi^2 + \int_{0}^{\infty}\left(\mathrm{e}^{i\xi x} - 1 -\frac{i\xi x}{x^2 + 1}\right)\nu_n(dx), \label{} \\
			\psi(\xi) &= \log E[\mathrm{e}^{i \xi X}] = i c \xi -\frac{1}{2}a\xi^2 + \int_{0}^{\infty}\left(\mathrm{e}^{i\xi x} - 1 -\frac{i\xi x}{x^2 + 1}\right)\nu(dx). \label{}
		\end{align}
		for constants $\xi \in \bR$, $c_n, c \in \bR$, $a_n,a \geq 0$ and measures $\nu_n$ and $\nu$.  Here $\nu_n$ and $\nu$ are L\'evy measures of $X_n$ and $X$. It is well-known that $X_n \xrightarrow[n \to \infty]{d} X$ is equivalent to hold the following (see e.g. \cite{Kallenberg:Foundationsof}):
		\begin{align}
			\lim_{n \to \infty}c_n &= c, \label{eq79} \\
			\lim_{n \to \infty}\int_{[0,\infty)}f(x) \tilde{\nu_n}(dx) &= \int_{[0,\infty)}f(x) \tilde{\nu}(dx) \ (\forall f \in C_b(\bR_+)) \label{eq77}  
		\end{align}
		where $C_b(\bR_+)$ denotes a space of bounded continuous functions on $[0,\infty)$ and
		\begin{align}
			\tilde{\nu_n}(dx) = a_n \delta_0(dx) + \frac{x^2}{x^2 + 1}\nu_n(dx)
		\end{align}		
		and $\tilde{\nu}$ is defined similarly. We note that the Laplace exponent $\chi_n(\lambda)$ of $X_n$ is 
		\begin{align}
			\chi_n(\lambda) = -\psi_n(i\lambda) = c_n\lambda - \int_{[0,\infty)}\left( \mathrm{e}^{-\lambda x} -1  + \frac{\lambda x}{x^2 +1} \right)\frac{x^2 + 1}{x^2}\tilde{\nu_n}(dx)\ (\lambda \geq 0). \label{eq78}
		\end{align}
		Then we have
		\begin{align}
			\chi_n(\lambda + 2) - 2\chi_n(\lambda + 1) + \chi_n(\lambda)
			= -\int_{[0,\infty)}\mathrm{e}^{-\lambda x}(\mathrm{e}^{-x} - 1)^2\frac{x^2 + 1}{x^2}\tilde{\nu_n}(dx). \label{} 				
		\end{align}
		From the assumption, it holds that $\lim_{n \to \infty}\chi_n(\lambda) = \chi(\lambda)$ for every $\lambda > 0$. Then we have
		\begin{align}
			\lim_{n \to \infty}\int_{[0,\infty)}\mathrm{e}^{-\lambda x}(\mathrm{e}^{-x} - 1)^2\frac{x^2 + 1}{x^2}\tilde{\nu_n}(dx) = \int_{[0,\infty)}\mathrm{e}^{-\lambda x}(\mathrm{e}^{-x} - 1)^2\frac{x^2 + 1}{x^2}\tilde{\nu}(dx). \label{}
		\end{align} 
		Hence by the continuity theorem for Laplace transform, we obtain
		\begin{align}
			\lim_{n \to \infty}\int_{[0,\infty)}f(x)(\mathrm{e}^{-x} - 1)^2\frac{x^2 + 1}{x^2}\tilde{\nu_n}(dx) = \int_{[0,\infty)}f(x)(\mathrm{e}^{-x} - 1)^2\frac{x^2 + 1}{x^2}\tilde{\nu}(dx) \label{eq156}
		\end{align}
		for every $f \in C_b(\bR_+)$. Since the function $(\mathrm{e}^{-x} - 1)^2\frac{x^2 + 1}{x^2}$ is bounded and the infimum is greater than $0$ on $[0,\infty)$, we obtain \eqref{eq77}. Then from \eqref{eq78} and \eqref{eq156}, we can easily deduce \eqref{eq79}, and the proof is complete.
	\end{proof}

\section{Appendix: Convergence of occupation times to non-degenerate distributions}\label{appendix: non-degenerate case}
 Here we treat the case $\frac{1}{t}A(t)$ converges in law to non-degenerate distribution.
 As we mentioned in Section \ref{section: intro}, we can apply the methods used in \cite{Watanabe:Arcsinelaw}, which are double Laplace transforms and Williams formula.
 
 The following proposition called Williams formula can be proved by almost the same way in \cite{IkedaWatanabe:Stochastic}. Though it is shown only for the Brownian motion in \cite{IkedaWatanabe:Stochastic}, the proof is essentially due to the property of the process that it does not jump from the positive (negative) side to the negative (positive) side without visiting the origin, and therefore we can extend the result to our situation. 
 \begin{Prop}[Williams formula]\label{Williamsformula}
 	Let $m_+,m_- \in \cM_1$ and $j_+,j_-$ be Radon measures on $(0,\infty)$ and suppose $(m_+,j_+)$ and $(m_-,j_-)$ satisfy (C). Then the following hold:
 	\begin{align}
 	A^{-1}(t) = t + \eta_{m_-,j_-}(\eta_{m_+,j_+}^{-1}(t)), \label{}
 	\end{align}
 	Here $A^{-1}(t)$ and $\eta_{m_+,j_+}^{-1}(t)$ are the right-continuous inverse processes of $A(t)$ and $\eta_{m_+,j_+}(t)$, respectively.
 \end{Prop}  
 
 We can also prove the following two propositions by the same argument in \cite{Watanabe:Arcsinelaw}. For the proof of the former, Williams formula is utilized.
\begin{Prop} \label{repofDLT}
	Under the same assumption in Proposition \ref{Williamsformula}, we have for every $\lambda > 0$ and $\mu > 0$,
	\begin{align}
	\int_{0}^{\infty}\mathrm{e}^{-\mu t}E[\mathrm{e}^{-\lambda A(t)}]dt 
	= \frac{\chi_{m_+,j_+}(\lambda + \mu) / (\lambda + \mu) + \chi_{m_-,j_-}(\mu)/ \mu}{\chi_{m_+,j_+}(\lambda + \mu) + \chi_{m_-,j_-}(\mu)}. \label{}
	\end{align}
\end{Prop}
\begin{Prop}\label{continuityofDLT}
	Under the same assumption in Proposition \ref{Williamsformula},	let $\zeta$ be a real-valued random variable. Then $\frac{1}{t}A(t) \xrightarrow[t \to \infty]{d} \zeta$ is equivalent to the following holds:
	\begin{align}
	\lim_{\gamma \to \infty}\int_{0}^{\infty}\mathrm{e}^{-\mu t}E[\mathrm{e}^{-\frac{\lambda}{\gamma} A(\gamma t)}]dt 
	= \int_{0}^{\infty}\mathrm{e}^{-\mu t}E[\mathrm{e}^{-\lambda t\zeta}]dt
	\ \text{for every} \ \lambda , \mu > 0. \label{}
	\end{align}
\end{Prop}
\begin{Def}
	For $\alpha,p \in [0,1]$, we define the generalized arcsine distribution $\mu_{\alpha,p}$ which is characterized by its Stieltjes transform:
	\begin{align}
	\int_{0}^{\infty}\frac{\mu_{\alpha,p}(dx)}{\lambda + x} = \frac{p(\lambda +  1)^{\alpha - 1} + (1-p)\lambda^{\alpha-1}}{p(\lambda + 1)^{\alpha} + (1-p)\lambda^{\alpha}} \ (\lambda > 0).  \label{}
	\end{align}
\end{Def}

 The following theorem can be proved by the same argument in Theorem \ref{alpha1-2}.
\begin{Thm}\label{alpha0-1}
	Let $m \in \cM_1$, $j$ be a Radon measure on $(0,\infty)$ and  $K$ be a slowly varying function at $\infty$ and, assume $(m,j)$ satisfies $\mathrm{(C)}$.
	Suppose the following hold: 
	\begin{enumerate}
		\item $m(x) \sim (1 - \alpha)^{-1}x^{1/\alpha - 1}K(x) \ (x \to \infty)$ for a constant $\alpha \in (0,1)$,
		\item $\kappa := \int_{0}^{\infty}xj(dx) < \infty$.
	\end{enumerate} 
	Then we have
	\begin{align}
	&\frac{1}{\gamma^{1/\alpha}K(\gamma)}\eta_{m,j}(\gamma t)  \xrightarrow[\gamma \to \infty]{d}T(m^{(\alpha)};\kappa t) \ \text{on}\ \bD, \label{} \\
	&n_{m,j}[T_0 > s] \sim \frac{\kappa\alpha^{\alpha-1}}{\Gamma(\alpha)} s^{-\alpha}L^{\sharp}(s)^{-\alpha} \ (s \to \infty), \label{eq110}
	\end{align}
	where $L^{\sharp}(x)$ be a de Bruijn conjugate of $L(x) = K(x^{\alpha})$.
\end{Thm}

The following is the desired limit theorem.
\begin{Thm}\label{arcalpha0-1}
	Let $m_+,m_- \in \cM_1$ and $j_+,j_-$ be Radon measures on $(0,\infty)$ and suppose $(m_+,j_+)$ and $(m_-,j_-)$ satisfy (C).
	Assume the following hold:
	\begin{enumerate}
		\item $m_\pm(x) \sim c_\pm (1-\alpha)^{-1} x^{1/\alpha -1}K (x) (x \to \infty)$ for constants $\alpha  \in (0,1)$, $c_\pm > 0$  and a slowly varying function $K$ at $\infty$, respectively,
		\item $\kappa_\pm := \int_{0}^{\infty}xj_\pm(dx) < \infty$. 
	\end{enumerate}
	Then we have
	\begin{align}
	\frac{1}{t}A(t) \xrightarrow[t \to \infty]{d} Y_{\alpha,p}, \label{}
	\end{align}
	where $p = \frac{\kappa_+c_+^\alpha}{\kappa_+c_+^\alpha + \kappa_-c_-^\alpha}$ and $Y_{\alpha,p}$ is distributed as $\mu_{\alpha,p}$.
\end{Thm}


\end{document}